\documentclass[10pt,a4paper]{amsart}
\usepackage[T1]{fontenc}
\usepackage{lmodern}
\usepackage[utf8]{inputenc}
\usepackage{amsmath,amssymb,mathtools,mathrsfs}
\usepackage{tikz,tikz-cd,enumitem}
\usepackage[margin=1in]{geometry}
\usepackage{microtype,needspace}
\usepackage[hidelinks]{hyperref}
\hypersetup{pdftitle={On the Hodge theory of toroidal embeddings and corresponding vanishings},pdfauthor={Chuanhao Wei}}
\newtheorem{theorem}{Theorem}[section]
\newtheorem{proposition}[theorem]{Proposition}
\newtheorem{lemma}[theorem]{Lemma}

\theoremstyle{definition}
\newtheorem{definition}[theorem]{Definition}

\newtheorem{example}[theorem]{Example}
\numberwithin{equation}{section}
\newcommand{\C}{\mathbb C}
\newcommand{\Q}{\mathbb Q}
\newcommand{\R}{\mathbb R}
\newcommand{\Z}{\mathbb Z}
\newcommand{\cO}{\mathcal O}
\newcommand{\cR}{\mathcal R}
\DeclareMathOperator{\DR}{DR}
\DeclareMathOperator{\Gr}{Gr}
\DeclareMathOperator{\Spec}{Spec}
\DeclareMathOperator{\Proj}{Proj}
\DeclareMathOperator{\Cone}{Cone}
\DeclareMathOperator{\Hom}{Hom}
\DeclareMathOperator{\Span}{Span}
\DeclareMathOperator{\MHM}{MHM}
\DeclareMathOperator{\id}{id}
\DeclareMathOperator{\reg}{reg}
\newcommand{\dR}[3]{\DR^{\bullet}_{(#1,!#2,*#3)}}
\newcommand{\Om}[4]{\Omega^{#1}_{(#2,!#3,*#4)}}
\newcommand{\Tw}[5]{\bigl(\Omega^{#1}_{(#2,!#3,*#4)}\otimes\cO_{#2}(#5)\bigr)^{\vee\vee}}
\newcommand{\tDR}{\widetilde{\DR}}
\setlist[enumerate]{label=\textup{(\arabic*)},leftmargin=2.2em}
\setlist[itemize]{leftmargin=2em}
\title{On the Hodge theory of toroidal embeddings and corresponding vanishings}
\author{Chuanhao Wei}
\date{}
\begin{document}
\begin{abstract}
We prove $E_1$-degeneration of the logarithmic Hodge--de Rham spectral sequence for proper toroidal triples. For a compatible Weil divisor, we realize the associated twisted logarithmic de Rham complex as the filtered de Rham complex of an object in the bounded derived category of mixed Hodge modules. We use this realization to prove a vanishing theorem of Kawamata--Viehweg type for toroidal triples. We also prove relative Bott vanishing in the toric case. On the local toric charts of a toroidal triple $(X,!B,*C)$, we characterize when the canonical exchange morphism between the two ordered extensions of the constant sheaf on $X\setminus(B+C)$ is an isomorphism, when each order recovers the logarithmic de Rham complex, and when this complex is perverse.
\end{abstract}
\maketitle
\section{Introduction}\label{sec:intro}

Let $X$ be an algebraic variety over $\C$, and let $B$ and $C$ be two
reduced Weil divisors with no common irreducible components. Set
$D=B+C$, $U=X\setminus D$, $X_B=X\setminus C$, and $X_C=X\setminus B$.
Consider the open immersions
\[
\begin{tikzcd}
U\arrow[r,"i_B"]\arrow[d,"i_C"'] & X_B\arrow[d,"j_C"]\\
X_C\arrow[r,"j_B"'] & X.
\end{tikzcd}
\]
For a constructible complex $K$ on $X^{\mathrm{an}}$, write $K[!B]=j_{B,!}j_B^*K$ for extension by zero across $B$ and $K[*C]=Rj_{C,*}j_C^*K$ for derived direct image across $C$. We read the operations from left to right. The two paths give
\[
 \C_X[!B][*C]=Rj_{C,*}i_{B,!}\C_U,\qquad
 \C_X[*C][!B]=j_{B,!}Ri_{C,*}\C_U.
\]
We shift de Rham complexes by $[n]$ for $n=\dim X$. Constructible sheaves are taken on $X^{\mathrm{an}}$ and coherent cohomology on $X$.

\begin{theorem}[Logarithmic comparison]
Assume that $X$ is smooth and $D=B+C$ is a normal crossing divisor.
We have the following quasi-isomorphisms in the constructible derived category $D^b_c(X^{\mathrm{an}},\C)$
\[
\C_X[!B][*C][n]
\simeq \C_X[*C][!B][n]
\simeq \DR_X^\bullet(\log D)(-B),
\]
where
\[
\DR_X^\bullet(\log D)(-B)
=\bigl[\cO_X(-B)\longrightarrow\Omega_X^1(\log D)(-B)
\longrightarrow\cdots\longrightarrow\Omega_X^n(\log D)(-B)\bigr][n].
\]
\end{theorem}

\begin{theorem}[Hodge--de Rham degeneration]
Assume that $X$ is smooth and proper and $D=B+C$ has normal crossings. The filtration by form degree gives a spectral sequence
\[
E_1^{p,q}=H^q(X,\Omega_X^p(\log D)(-B))
\Longrightarrow \mathbb H^{p+q}(X,\C_X[*C][!B])
\]
which degenerates at $E_1$. Either extension order gives the same abutment by logarithmic comparison.
\end{theorem}

For these classical results, see~\cite[Remark~4.2.2(c) and Corollary~4.2.4]{DI87}, and for $B=0$, see~\cite[Proposition~3.1.8 and Corollary~3.2.13(ii)]{Del71}. Logarithmic comparison extends to mixed Hodge modules with compatible $V$-filtrations~\cite[Theorem~1]{Wei20} and to good mixed twistor $\mathscr D$-modules~\cite[Corollary~3.9 and Proposition~3.14]{Wei23b}.

In this paper, we study the following question. To what extent do these comparison and degeneration theorems extend to toroidal varieties?

\begin{definition}\label{intro:triple}
Let $X$ be a normal toric variety, and let $B$ and $C$ be reduced
torus-invariant Weil divisors with no common irreducible components.
We call $(X,!B,*C)$ a \emph{toric triple}. It is \emph{plenary} if
$B+C$ is the whole torus-invariant boundary.
A triple $(X,!B,*C)$ on a normal variety is \emph{toroidal} if every point $x\in X$ admits \'etale morphisms $p\colon(V,v)\to(X,x)$ and $q\colon(V,v)\to(T_x,t)$, where $(T_x,!B_x,*C_x)$ is an affine toric triple, $t$ lies on its closed orbit, $p^*B=q^*B_x$, and $p^*C=q^*C_x$. We call such a diagram a \emph{toric chart}.
\end{definition}

We first consider degeneration. On a toroidal triple $(X,!B,*C)$, we extend $\Omega^p_X(\log(B+C))(-B)$ reflexively from the smooth normal crossing locus. The resulting sheaves $\Om{p}{X}{B}{C}$ are Danilov's logarithmic $p$-forms. They form the logarithmic de Rham complex $\dR{X}{B}{C}$.

\begin{theorem}\label{thm:degeneration}
Let $(X,!B,*C)$ be a toroidal triple with $X$ proper of dimension $n$. The spectral sequence
\[
 E_1^{p,q}=H^q(X,\Om{p}{X}{B}{C})\quad\Longrightarrow\quad
 \mathbb H^{p+q-n}(X,\dR{X}{B}{C})
\]
degenerates at $E_1$.
\end{theorem}

For $B=0$ and $C$ Cartier, Danilov proved logarithmic comparison and conjectured this degeneration~\cite[\S15.9 and Lemma~15.10]{Dan78}.
The case $B=C=0$ was proved by cubical hyperresolutions in
\cite{GNAPGP88}, and Danilov gave a more elementary proof in \cite{Dan91}. For $C=0$ and arbitrary $B$, the degeneration follows from Ambro's theorem on relative differential forms for weakly toroidal pairs~\cite{Amb18}.
Felten, Filip, and Ruddat proved the case $B=0$ for arbitrary reduced toroidal boundary $C$ by reduction to positive characteristic~\cite[Theorem~1.4]{FFR21}. For complete toric varieties, Fujino proved degeneration with both $B$ and $C$ over a field of arbitrary characteristic~\cite[Theorem~1.4]{Fuj08}. His proof uses multiplication by $p$ on the torus in characteristic $p$ and passes to characteristic zero by reduction modulo~$p$ and base change~\cite[Proposition~3.3 and proof of Theorem~2.3.7]{Fuj08}.

Our aim is to prove degeneration and vanishing by Hodge theory on toroidal triples without a global torus action. Let $f\colon Y\to X$ be the canonical modification associated to $C-B$, with strict transforms $B_Y,C_Y$. On $U_Y=Y\setminus(B_Y+C_Y)$, take the shifted constant complex $\C_{U_Y}^H[n]\in D^b\MHM(U_Y)$. Extend it first by $*$ across $C_Y$ and then by $!$ across $B_Y$. By Theorem~\ref{thm:hodge-realization}, the proper direct image of this ordered extension under $f$ has filtered de Rham complex $\dR{X}{B}{C}$. For proper $X$, Saito's strictness theorem gives degeneration. Since $Y$ is then proper, extension by zero across $B_Y$ gives, for $i\colon U_Y\hookrightarrow Y\setminus B_Y$ and $k=p+q$,
\[
 \mathbb H^{k-n}\bigl(X,\dR{X}{B}{C}\bigr)
 \simeq \mathbb H_c^k\bigl(Y\setminus B_Y,Ri_*\C_{U_Y}\bigr).
\]

We next consider vanishing for logarithmic forms twisted by a Weil divisor $L$. We tensor with the rank-one reflexive sheaf $\cO_X(L)$ and take the reflexive hull~\cite[Definition~3.3 and Proposition~3.4]{Sch}. For $D=\sum_i D_i$, we write $\mathbf dD=\sum_i d_iD_i$, with rational coefficients.

In the smooth projective toric case, the positive vanishing in Theorem~\ref{thm:positive}\textup{(\ref{intro:vanishing-positive})} follows from~\cite[Theorem~1.1]{Wei23a} by taking the boundary $B+C$, the line bundle $\cO_X(L+C)$, and coefficients $\mathbf b$ on $B$ and $1-\mathbf c$ on $C$. In~\cite[Theorem~1.5]{Wei23a} and~\cite[Theorem~2.9]{Wei23b}, we prematurely announced, without proof, negative Kawamata--Viehweg vanishing for toroidal triples and negative Bott vanishing for toric triples. Example~\ref{ex:negative-cone} disproves both assertions, even when $B=C=0$ and $L$ is Cartier. Part~\textup{(\ref{intro:vanishing-negative})} gives a corrected statement under hypotheses on the boundary coefficients and on the Cartier locus of $L$.

\Needspace{9\baselineskip}
\begin{theorem}\label{thm:positive}
Let $(X,!B,*C)$ be a proper toroidal triple of dimension $n$, let $A$ be an ample $\Q$-Cartier divisor, and let $L$ be a Weil divisor.
\begin{enumerate}[ref=\arabic*]
\item\label{intro:vanishing-positive} If $L\sim_{\Q} A+\mathbf{b}B-\mathbf{c}C$ with $0\leq b_i,c_j\leq1$, then
\begin{equation*}
 H^q\bigl(X,\Tw{p}{X}{B}{C}{L}\bigr)=0
 \qquad\text{for }p+q>n.
\end{equation*}
If the triple is toric, the vanishing holds for every $p\geq0$ and $q>0$.
\item\label{intro:vanishing-negative} Suppose $L\sim_{\Q}-A+\mathbf bB-\mathbf cC$ with $0<b_i,c_j<1$. Let $U_L$ be the Cartier locus of $L$ in $U=X\setminus(B+C)$. If $\C_{U_L}[n]$ is perverse on $U_L$, then
\[
 H^q\bigl(X,\Tw{p}{X}{B}{C}{L}\bigr)=0
 \qquad\text{for }p+q<n.
\]
\end{enumerate}
\end{theorem}

If $\C_U[n]$ is perverse, then its restriction to each $U_L$ is perverse. Thus the perversity hypothesis in part~\textup{(\ref{intro:vanishing-negative})} holds for every $L$, for example when $U$ is a local complete intersection~\cite[Example~2.19(d)]{MS21}. Proposition~\ref{cohom:perversity} gives a finite test for perversity of $\C_{U_L}[n]$ on each toric chart.

We return to logarithmic comparison. The proof of degeneration uses boundary extensions on the modification $Y$. When do the two orders of extension on $X$ itself agree, and when does each recover the logarithmic complex? Both questions can be tested on the toric charts of a toroidal triple, so we state the criteria for toric triples.

For $T=\C_X[!B][*C]$, open base change gives $j_B^*T\simeq Ri_{C,*}\C_U$. Hence the map in the following theorem is the counit $j_{B,!}j_B^*T\to T$.

\begin{theorem}[Exchange of boundary extensions]\label{thm:exchange}
Let $(X,!B,*C)$ be a toric triple with fan $\mathfrak X$. The canonical exchange morphism
\[
 \C_X[*C][!B]\longrightarrow\C_X[!B][*C]
\]
is an isomorphism if and only if the triple satisfies the \emph{exchange condition} of Definition~\ref{def:exchange}. This condition depends only on the face poset of $\mathfrak X$ and the rays marked by $B$ and $C$.
\end{theorem}

When exchange is an isomorphism, write $\C_X[!B+*C]$ for the common extension. A toric triple is partially sorted if, in every cone with no $B$-ray, the unmarked rays are precisely the rays of a face. It is geometrically sorted if it is partially sorted and satisfies the exchange condition (Definition~\ref{def:geometric-sorted}).

\begin{theorem}\label{thm:circuit-comparison}
Let $(X,!B,*C)$ be a toric triple of dimension $n$. Consider the canonical comparison followed by the exchange morphism of Theorem~\ref{thm:exchange}, shifted by $[n]$,
\[
 \dR{X}{B}{C}\xrightarrow{\ u\ }\C_X[*C][!B][n]
 \xrightarrow{\ e\ }\C_X[!B][*C][n].
\]
For every such triple, $u$ is an isomorphism on stalks along $B$. It is an isomorphism on $X$ if and only if the triple is partially sorted. The composite $e\circ u$ is an isomorphism if and only if the triple is geometrically sorted.
\end{theorem}

Section~\ref{sec:prelim} sets up logarithmic forms and their complexes on toroidal triples, including compatible Weil twists. We recall the mixed Hodge module formalism and deduce filtered comparison on simplicial charts from the smooth case by finite descent. For a general triple, Section~\ref{sec:fans} constructs the canonical modification associated to $C-B$ and refines its charts until the marked divisors become coordinate hyperplanes on finite covers. Section~\ref{sec:bott} uses restriction to unmarked invariant divisors to prove relative Bott vanishing. Under these modifications, each sheaf of logarithmic forms thus pushes forward to the original sheaf with vanishing higher direct images.

In Section~\ref{sec:proofs}, we use this direct-image formula to construct the Hodge realizations of the logarithmic complex and its compatible twists. Proper strictness gives degeneration. We prove Kawamata--Viehweg vanishing by induction on dimension, using Artin vanishing on the affine complement of a general hyperplane section and proper strictness for the twisted complex. Section~\ref{sec:circuits} proves the exchange criterion independently of Hodge theory and gives the criteria for logarithmic comparison and perversity.

\subsection*{Acknowledgment}
The author would like to thank Junliang Shen for bringing the Bott
Vanishing Theorem to his attention and Burt Totaro for his encouragement
to explore a geometric proof of the Bott Vanishing Theorem for singular
toric varieties, which ultimately led to this paper. The author also
appreciates the insightful discussions with Lifan Guan and Lei Wu during
the preparation of this work.

The author prepared the first version of this paper without AI assistance and submitted it to arXiv on 13 October 2024 (\href{https://arxiv.org/abs/2410.09899v1}{arXiv:2410.09899v1}).
In revising the paper, the author used OpenAI's ChatGPT (GPT-6) .
GPT-6 contributed the counterexample in Example~\ref{ex:negative-cone} and the perversity criterion in Proposition~\ref{cohom:perversity}.
It also helped formulate the equivalence in Theorem~\ref{thm:exchange} and the perversity hypothesis in Theorem~\ref{thm:positive}\textup{(\ref{intro:vanishing-negative})}, revise proofs, and reorganize the exposition.
The author directed these revisions, checked the counterexample and the proofs of the stated results, and takes responsibility for the final manuscript.

During the preparation of this paper, the author is supported by the National Key R\&D Program of China (\#2024YFA1014400).
\section{Preliminaries}\label{sec:prelim}

Let $X$ be a normal toric variety of dimension $n$. Fix a lattice
$N\simeq\Z^n$ and put $M=\Hom(N,\Z)$. Write $N_{\Bbbk}=N\otimes_{\Z}\Bbbk$ and $M_{\Bbbk}=M\otimes_{\Z}\Bbbk$ for $\Bbbk=\Q,\R,\C$, and let $\mathfrak{X}$ be the fan of $X$ in $N_\R$. Write $\mathfrak{X}(1)$ for its set of rays.
A fan triple $(\mathfrak{X},!\mathfrak{B},*\mathfrak{C})$ consists of the fan and two disjoint ray sets $\mathfrak{B},\mathfrak{C}\subseteq\mathfrak{X}(1)$. The corresponding reduced invariant divisors give the toric triple $(X,!B,*C)$.
Let $\mathfrak{A}$ be the set of unmarked rays, the complement of $\mathfrak{B}\cup\mathfrak{C}$ in $\mathfrak{X}(1)$.
For a cone $\sigma\in\mathfrak{X}$, write $\sigma(1)$ for its rays. We use a subscript for restriction to a cone, so $\mathfrak{E}_{\sigma}=\mathfrak{E}\cap\sigma(1)$ for any set of rays $\mathfrak{E}$. By convention, the empty set of rays generates the zero face.
A subdivision is regarded as a map of fans by sending a cone to
the smallest cone of the original fan containing it.

A toroidal triple is \emph{simplicial} if every cone in its local toric charts is simplicial.

For $\sigma\in\mathfrak{X}$, put $d=\dim\sigma$ and
\[
 \begin{gathered}
 N_\sigma=N\cap\Span_\R\sigma,\qquad M_\sigma=\Hom(N_\sigma,\C),\\
 S_\sigma=\Spec\C[\{m\in\Hom(N_\sigma,\Z):\langle m,v\rangle\geq0\text{ for all }v\in\sigma\}].
 \end{gathered}
\]
The variety $S_\sigma$ has dimension $d$ and a unique torus-fixed point $x_\sigma$.
Since $N_\sigma$ is saturated, a splitting of $N_\sigma\subset N$ gives
\[
 U_\sigma:=\Spec\C[\sigma^\vee\cap M]\simeq S_\sigma\times O(\sigma),
 \qquad O(\sigma)\simeq(\C^*)^{n-d}.
\]
Under this isomorphism, the orbit $O(\sigma)$ is $\{x_\sigma\}\times O(\sigma)$. We call $(S_\sigma,x_\sigma)$ the \emph{transverse toric slice} to this orbit. A choice of splitting identifies $S_\sigma\times\{x\}$ with a slice through any $x\in O(\sigma)$. The analytic germ of $U_\sigma$ at $x$ is $(S_\sigma,x_\sigma)\times(\C^{n-d},0)$, with both marked divisors pulled back from $S_\sigma$. Over $\C$, a common \'etale chart as in Definition~\ref{intro:triple} identifies the marked analytic germ of a toroidal triple with that of its toric model. Conversely, an analytic isomorphism of marked germs with a toric model gives a common \'etale chart by Artin approximation applied with the ideals of $B$ and $C$~\cite[Corollary~2.1 and the proof of Corollary~2.6]{Art69}.

\subsection{Logarithmic forms of toric triples}

Let $(X,!B,*C)$ be a toroidal triple. Choose an open immersion
$j\colon X^\circ\hookrightarrow X$ whose complement has codimension at least two,
such that $X^\circ$ is smooth and $(B+C)|_{X^\circ}$ has simple normal
crossings. On $X^\circ$, we use $B$ and $C$ for their restrictions. We define
\begin{equation}\label{pre:intrinsic-forms}
 \Om{p}{X}{B}{C}
 :=j_*\bigl(\Omega^p_{X^\circ}(\log(B+C))\otimes
                   \cO_{X^\circ}(-B)\bigr).
\end{equation}
For the affine toric model $X=U_\tau=\Spec\C[\tau^\vee\cap M]$, let $v_\rho$ be the primitive generator of a ray $\rho$, and $x^m$ the character of $m\in M$. Write $\tau^\vee(-B)$ for the set of $m\in\tau^\vee$ with $\langle m,v_\rho\rangle>0$ for every $\rho\in\mathfrak B_\tau$. For a face $\zeta\preceq\tau$ containing no $\mathfrak B$-ray, put
\[
 \phi_{(X,!B,*C)}(\zeta)=\bigcap_{\rho\in\mathfrak A_\zeta}v_\rho^\perp\subset M_\C,
\]
with the empty intersection equal to $M_\C$.

\begin{proposition}\label{prop:reflexive}
The sheaf~\eqref{pre:intrinsic-forms} is coherent and reflexive. Its formation is independent of $X^\circ$ and commutes with \'etale pullback. On an affine toric chart, it has the character decomposition
\begin{equation}\label{pre:weight-formula}
 \Gamma(X,\Om{p}{X}{B}{C})
 =\bigoplus_{m\in\tau^\vee(-B)\cap M}
       \bigwedge^p\!\phi_{(X,!B,*C)}(\tau\cap m^\perp)\,x^m
 \ \subset\ \C[M]\otimes\bigwedge^p M_\C.
\end{equation}
This formula commutes with restriction to affine face charts. In particular, $\Om{0}{X}{B}{C}=\cO_X(-B)$.
\end{proposition}

\begin{proof}
On a normal variety, a locally free sheaf on $X^\circ$ has a unique coherent reflexive extension, and this extension commutes with \'etale pullback. Apply Danilov's formula~\cite[\S15.5]{Dan78} with logarithmic boundary $B+C$. At a ray of $B$, the twist by $-B$ requires each logarithmic coefficient to vanish to order at least one. This replaces $\langle m,v_\rho\rangle\geq0$ by $\langle m,v_\rho\rangle>0$ at precisely those rays and gives~\eqref{pre:weight-formula}. Restriction to a face retains the conditions at its rays.
\end{proof}

The smooth toric proof of Bott vanishing uses restriction to an unmarked invariant divisor~\cite[\S2]{Wei23a}. The following face conditions suffice for restriction and residue on a toric triple.

\begin{proposition}\label{prop:restriction}
Let $i\colon E\hookrightarrow X$ be an unmarked invariant prime divisor of a toric triple $(X,!B,*C)$, with ray $\varepsilon$ and primitive generator $e$. Assume that $B_E=(B\cap E)_{\mathrm{red}}$ and $C_E=(C\cap E)_{\mathrm{red}}$ are divisors on $E$, possibly empty. Suppose that in each cone $\sigma$ containing $\varepsilon$ and no $\mathfrak B$-ray, the unmarked rays generate a face of $\sigma$. For every $p\geq0$, restriction gives an exact sequence
\begin{subequations}
\begin{equation}\label{pre:restriction-sequence}
 0\longrightarrow\Om{p}{X}{(B+E)}{C}
 \longrightarrow\Om{p}{X}{B}{C}
 \xrightarrow{\operatorname{restr}_E}i_*\Om{p}{E}{B_E}{C_E}\longrightarrow0.
\end{equation}
If the unmarked rays other than $\varepsilon$ also generate a face of each such $\sigma$, then residue gives an exact sequence
\begin{equation}\label{pre:residue-sequence}
 0\longrightarrow\Om{p}{X}{B}{C}
 \longrightarrow\Om{p}{X}{B}{(C+E)}
 \xrightarrow{\operatorname{Res}_E}
 i_*\Om{p-1}{E}{B_E}{C_E}\longrightarrow0,
\end{equation}
\end{subequations}
with zero residue term for $p=0$.
\end{proposition}

\begin{proof}
Work on an affine chart $U_\tau$ meeting $E$ and apply~\eqref{pre:weight-formula} in degree $m\in\tau^\vee\cap M$. If $\langle m,e\rangle>0$, the first two terms of each sequence have the same coefficient and the term on $E$ vanishes. Suppose $\langle m,e\rangle=0$. Put $\zeta=\tau\cap m^\perp$ and write $\bar\zeta$ for its image in $N_\R/\R e$. The orbit of $\zeta$ lies in $B\cap E$ exactly when $\zeta$ contains a $\mathfrak B$-ray. Since $B_E$ is divisorial or empty, this is equivalent to $\bar\zeta$ containing a ray of $B_E$. In this case, all coefficients vanish.

Otherwise, let $F$ be the face of $\zeta$ generated by its unmarked rays. The unmarked rays of $\bar\zeta$ are precisely the rays of the image of $F$ in $N_\R/\R e$. Indeed, both correspond to the two-dimensional faces of $F$ containing $\varepsilon$. Thus the coefficient on $E$ in~\eqref{pre:restriction-sequence} is $\bigwedge^p F^\perp$, viewed in $\bigwedge^p M_\C$. On $X$, the first coefficient is zero and the second is $\bigwedge^p F^\perp$. Restriction identifies the latter with the coefficient on $E$.

Under the additional face hypothesis, let $G$ be the face of $\zeta$ generated by the unmarked rays other than $\varepsilon$. Since $G$ does not contain $\varepsilon$, choose a supporting linear form $u$ for $G\preceq\zeta$ with $u(e)=1$. Then $G^\perp=F^\perp\oplus\C u$. The residue map on the coefficient $\bigwedge^p G^\perp$ is contraction by $e$. Its kernel is $\bigwedge^p F^\perp$. For $p\geq1$, every target coefficient $\beta\in\bigwedge^{p-1}F^\perp$ is the residue of $u\wedge\beta$. This proves exactness of~\eqref{pre:residue-sequence}. The restriction and residue maps commute with localization on face charts and hence glue.
\end{proof}

The next proposition identifies the direct images of twisted logarithmic forms under a small toroidal modification. For a compatible Weil divisor twist, these sheaves are the terms of the twisted logarithmic de Rham complex~\eqref{intro:twisted-dr}.

\begin{proposition}\label{prop:pushforward}
Let $f\colon X'\to X$ be a proper small toroidal modification, with
strict-transform markings $B',C'$. For a Weil divisor $L$ on $X$,
write $L'$ for its strict transform. Then
\[
 f_*\Tw{p}{X'}{B'}{C'}{L'}\simeq\Tw{p}{X}{B}{C}{L}.
\]
\end{proposition}

\begin{proof}
The two sheaves agree on the common isomorphism locus. Its complement has codimension at least two on both normal varieties. For every open $V\subseteq X$, sections over $f^{-1}(V)$ and over $V$ are therefore identified by reflexive extension from this common locus. For $L=0$, these identifications commute with exterior differentiation, since the maps agree with pullback on the common open and sections of the target are determined there.
\end{proof}

\subsection{Logarithmic and twistor de Rham complexes}\label{pre:twistor-dr}

On the torus, exterior differentiation sends $x^m\alpha$ to $x^m(m\wedge\alpha)$. It preserves the summands in~\eqref{pre:weight-formula}, since $m$ annihilates every ray of $\tau\cap m^\perp$. We thus obtain
the complex
\[
 \dR{X}{B}{C}
 =\bigl[\cO_X(-B)\longrightarrow\Om{1}{X}{B}{C}
       \longrightarrow\cdots\longrightarrow\Om{n}{X}{B}{C}\bigr][n],
\]
with $\Om{p}{X}{B}{C}$ placed in degree $p-n$. On the toric model of a chart, it is a subcomplex of the
logarithmic de Rham complex for the full invariant boundary. The intrinsic description
extends the differential to the toroidal case.

On the transverse toric slice $S_\sigma$, invariant forms compute the hypercohomology of the unshifted logarithmic complex on each affine face chart. To see this also near the vertex, let $V$ range over the neighborhoods of $x_\sigma$ cut out by upper bounds on the absolute values of monomial generators. These form a Stein neighborhood basis by the closed embedding defined by the generators. For $\eta\preceq\sigma$, write $U_{\sigma,\eta}$ for the affine chart of $\eta$ in the lattice $N_\sigma$. Both $U_{\sigma,\eta}$ and $V\cap U_{\sigma,\eta}$ are Stein, since the face chart is a principal open in $S_\sigma$. Taking $\eta=\sigma$ includes $V$ itself.

We adapt the contraction in~\cite[proofs of Proposition~13.4 and Lemma~15.10]{Dan78} to these opens. On their intersections with the dense torus, logarithmic forms have normally convergent Laurent expansions. Fix a basis $e_1,\ldots,e_d$ of $N_\sigma$. For each nonzero weight $m\in\Hom(N_\sigma,\Z)$, take the first $j$ with $\langle m,e_j\rangle\ne0$. Since the differential at $m$ is $m\wedge(-)$, the operator
\[
 h_m=\frac{\iota_{e_j}}{\langle m,e_j\rangle}
 \qquad\text{satisfies}\qquad dh_m+h_md=\id .
\]
The operators are uniformly bounded, since the denominators are nonzero integers. The contracted Laurent series therefore converges normally on the torus part of each open. At the open orbit of a boundary divisor, take a toric character $z$ of order one as transverse coordinate. In a logarithmic frame, the exponents of $z$ are nonnegative and are positive along $B$. At an unmarked ray, the coefficient of $z^0$ belongs to the exterior algebra of its annihilator. Contraction preserves these conditions in~\eqref{pre:weight-formula}. The Cauchy estimate on $|z|=r$, uniform on compact sets in the orbit directions, gives normal convergence on smaller transverse discs, and reflexivity extends the contracted form across the remaining strata of codimension at least two.

Setting $h_0=0$ gives a homotopy from the identity to projection onto weight zero. It commutes with restriction to face charts because the choice of $j$ depends only on $m$. On a finite toric lattice cover, deck translations preserve the monomial bounds and commute with the homotopy, since they multiply characters by roots of unity and fix invariant logarithmic forms. On each of these Stein opens, the inclusion of weight-zero forms, with zero differential, therefore computes hypercohomology by coherence of the logarithmic terms.

\begin{definition}\label{intro:quadruple}
A \emph{toric quadruple} $(X,!B,*C,!\!*\mathbf{h}H)$ consists of a toric
triple $(X,!B,*C)$, a reduced torus-invariant divisor $H$ having no common
components with $B+C$, and rational coefficients $0<h_i<1$.
A toric quadruple is \emph{plenary} if $B+C+H$ is the whole invariant boundary.
Toroidal quadruples are defined by common \'etale charts preserving $B$, $C$, and the rational divisor $\mathbf hH$. As for triples, this is equivalent to the analytic local definition. We call a toroidal triple or quadruple plenary if every point admits a plenary toric chart.
\end{definition}

For a toric quadruple $(X,!B,*C,!\!*\mathbf{h}H)$, we similarly write $(\mathfrak{X},!\mathfrak{B},*\mathfrak{C},!\!*\mathfrak{H})$ for the associated fan quadruple, with $\mathfrak{A}=\mathfrak{X}(1)\setminus(\mathfrak{B}\cup\mathfrak{C}\cup\mathfrak{H})$. In the product description of a toric chart in Section~\ref{sec:prelim}, the rational divisor $\mathbf hH$ is also pulled back from the transverse slice.

\begin{definition}\label{intro:compatible}
A Weil divisor $L$ is \emph{compatible} with $(X,!B,*C)$ if
\[
L\sim_{\Q}\mathbf{b}B-\mathbf{c}C,\qquad 0\leq b_i,c_j\leq1.
\]
Fix coefficients $\mathbf b,\mathbf c$ realizing this relation. On $B+C$, the divisor $C+\mathbf bB-\mathbf cC$ has coefficient $b_i$ on $B_i$ and $1-c_j$ on $C_j$. Let $F$, $H$, and $G$ be the reduced sums of the components on which this coefficient is $0$, belongs to $(0,1)$, and is $1$, respectively. Taking $\mathbf h$ to be the coefficients on $H$ gives the induced quadruple $(X,!F,*G,!\!*\mathbf hH)$ and
\[
 G+\mathbf hH=C+\mathbf bB-\mathbf cC\sim_\Q C+L.
\]
\end{definition}

For a toroidal triple and a compatible Weil divisor $L$ inducing $(X,!F,*G,!\!*\mathbf hH)$, write
\[
\Omega^p_{(X,!F,*G,!\!*\mathbf{h}H)}
=\Tw{p}{X}{B}{C}{L}.
\]
The sheaf notation suppresses $L$. To specify the differential, we also fix a cyclic cover and a character. Choose $m>0$ sufficiently divisible that $m\mathbf b,m\mathbf c$ are integral and $mL-m\mathbf bB+m\mathbf cC$ is principal, and choose $g\in\C(X)^*$ with this divisor. Fix the normalized cyclic cover $\pi\colon\widehat X\to X$ given by $t^m=g$ and the character $\chi(\zeta)=\zeta$ for the pullback action of $t\mapsto\zeta t$, $\zeta\in\mu_m$, on functions and forms. On the smooth normal crossing locus, this is the cyclic covering construction of~\cite[3.5 and 3.11]{EV92} used in~\cite[\S4]{Wei23b}. Normalization in the same finite algebra over $\C(X)$ defines the cover on all of $X$.

Since $dt/t=dg/(mg)$, differentiation of $at$ gives, for a rational section $a$ of $\cO_X(L-B)$,
\begin{equation}\label{intro:cyclic-connection}
 \nabla(a)=da+a\frac{dg}{mg}.
\end{equation}
In the logarithmic lattice $\cO_X(L-B)$ on the smooth locus, the residue along a component of $B+C$ is one minus its coefficient in $C+\mathbf bB-\mathbf cC$, by~\eqref{intro:cyclic-connection} (see also~\cite[3.16]{EV92}). The connection extends by the Leibniz rule and reflexivity to the logarithmic sheaves above, giving
\begin{equation}\label{intro:twisted-dr}
\DR^\bullet_{(X,!F,*G,!\!*\mathbf{h}H)}
=\bigl[\cO_X(-B+L)\longrightarrow
\Omega^1_{(X,!F,*G,!\!*\mathbf{h}H)}\longrightarrow\cdots
\longrightarrow\Omega^n_{(X,!F,*G,!\!*\mathbf{h}H)}\bigr][n].
\end{equation}

We keep this cover and character fixed throughout, since the boundary coefficients alone need not determine the monodromy.

We use the increasing ``stupid'' filtration on both complexes. For $0\le p\le n$ and $i\in\Z$, its degree-$p-n$ term in the untwisted case is
\begin{equation}\label{pre:filtration}
 \bigl(F_i\dR{X}{B}{C}\bigr)^{p-n}=
 \begin{cases}
 \Om{p}{X}{B}{C},&i\ge -p,\\
 0,&i<-p.
 \end{cases}
\end{equation}
Thus
\begin{equation}\label{pre:graded-shift}
 \Gr^F_{-p}\dR{X}{B}{C}=\Om{p}{X}{B}{C}[n-p].
\end{equation}

For any variety $X$, put $\mathcal X=X\times\mathbb A^1_z$ and let $\pi_X\colon\mathcal X\to\mathbb A^1_z$ be the projection.
The twistor de Rham complex is the Rees complex
$\tDR_{(X,!B,*C)}^\bullet=\bigoplus_{i\in\Z}F_i\dR{X}{B}{C}\,z^i$.
Its term in degree $p-n$ is $\Om{p}{X}{B}{C}z^{-p}\C[z]$, with generator of internal degree $-p$; in these frames the differential is $z\,d$. We write
$\mathcal E|_{z=a}=\mathcal E\otimes^L_{\C[z]}\C[z]/(z-a)$ for derived
specialization, and $(\mathcal E)_m$ for the component of internal degree $m$.
The terms are flat over $\C[z]$, and the derived fibers are
\begin{equation}\label{pre:rees-fibers}
 \tDR_{(X,!B,*C)}^\bullet\big|_{z=1}\simeq\dR{X}{B}{C},\qquad
 \tDR_{(X,!B,*C)}^\bullet\big|_{z=0}
 \simeq\bigoplus_p\Om{p}{X}{B}{C}[n-p].
\end{equation}
The fiber at $z=0$ retains its internal grading. At $z=1$, we first forget the grading, since $z-1$ is inhomogeneous. The same conventions apply to the twisted logarithmic complex.

A complex of Rees modules, or of sheaves of Rees modules, is called \emph{strict} if its cohomology is $z$-torsion-free. Suppose that $X$ is proper. If $R\pi_{X,*}\tDR_{(X,!B,*C)}^\bullet$ is strict, specialization at $z=0$ commutes with cohomology, and~\eqref{pre:rees-fibers} gives
\begin{equation}\label{pre:strictness-graded}
\Gr^F_{-p}\mathbb H^k\bigl(X,\dR{X}{B}{C}\bigr)
\simeq H^{k+n-p}\bigl(X,\Om{p}{X}{B}{C}\bigr),
\end{equation}
where hypercohomology carries the filtration induced by~\eqref{pre:filtration}. The same implication holds for the twisted complex. We prove strictness for both complexes in Section~\ref{sec:proofs}.

Rees comparisons take place in the derived
category of graded $\C[z]$-sheaves on $X^{\mathrm{an}}$, with morphisms of
internal degree zero. At $z=0$, we specify the induced $\cO_X$-linear
identifications of the homogeneous pieces.

We use the boundary-extension notation $\mathcal F[!D]$ and $\mathcal F[*D]$ of Section~\ref{sec:intro} for any reduced divisor $D$. We call $D$ \emph{clean} for
$\mathcal F$ if the canonical maps
$\mathcal F[!D]\to\mathcal F\to\mathcal F[*D]$ are both isomorphisms. On an open subset, each boundary symbol denotes the restriction of the corresponding divisor.
For an ordered list
$\mathscr{O}=(\diamond_1D_1,\ldots,\diamond_sD_s)$ of boundary components,
with $\diamond_a\in\{!,*\}$, set
$\mathcal F[\mathscr{O}]=\mathcal F[\diamond_1D_1]\cdots
[\diamond_sD_s]$. On a toroidal triple, the $D_a$ are irreducible components on $X$. If $D_a$ has several branches in a toric chart, the operation $[\diamond_aD_a]$ on that chart extends across their union, and its ray set consists of all the corresponding rays.

\subsection{Saito's mixed Hodge modules}\label{sec:saito}

We use Saito's formalism~\cite{Sai90,Sai13} and write $\MHM(X)$ for the category of graded-polarizable complex mixed Hodge modules~\cite[Appendix~A.1]{DV25}. The construction of the filtered de Rham complex on a singular variety uses local closed embeddings into smooth varieties and is independent of these choices.

For the structure map $a_Z\colon Z\to\mathrm{pt}$, write $\C_Z^H=a_Z^*\C_{\mathrm{pt}}^H\in D^b\MHM(Z)$. Its constructible realization is the constant sheaf $\C_Z$. When $Z$ is singular, $\C_Z^H[\dim Z]$ need not belong to $\MHM(Z)$.

The constructible realization takes cohomology in $D^b\MHM(X)$ to perverse cohomology in $D^b_c(X^{\mathrm{an}},\C)$. It detects isomorphisms and commutes with the six operations and their adjunction maps, since the operations in~\cite[Corollary~1]{Sai13} preserve the complex structure. If $\mu_m$ acts on $K$ and $\chi$ is a character of $\mu_m$, we denote by $K_\chi$ the summand on which each $\zeta^*$ acts by $\chi(\zeta)$.

Write $\tDR_X(\mathcal M)$ for the Rees de Rham complex of $\mathcal M\in D^b\MHM(X)$. Its fiber at $z=1$ is the underlying constructible complex. For a proper algebraic morphism $f\colon Y\to X$ and $\mathcal N\in D^b\MHM(Y)$, compatibility of the filtered de Rham functor with proper direct image~\cite[Theorem~0.1 and Proposition~2.8]{Sai00} gives, in Rees notation~\cite[Appendix~A]{ESY17},
\[
 \tDR_X(Rf_*\mathcal N)\simeq R\widetilde f_*\tDR_Y(\mathcal N),
 \qquad \widetilde f=f\times\id_{\mathbb A^1}\colon\mathcal Y\to\mathcal X.
\]
With our indexing, the $(-p)$-th graded piece of the de Rham complex of $\C_X^H[n]$ on a smooth $n$-fold is $\Omega_X^p[n-p]$.

Let $K$ be a filtered de Rham complex shifted by $[-n]$. Suppose that $\Gr^F_{-p}K$ has cohomology only in degree $p$ and is acyclic for $p\notin[0,n]$. Then $K$ is functorially represented by the complex of sheaves $\mathcal H^p(\Gr^F_{-p}K)$ with the first differential $d_1$ of the filtration and the filtration by degree~\cite[Section~3.1.7 and Proposition~3.1.8]{BBDG}. Indeed, $F_iK=0$ locally for $i\ll0$, so the vanishing of the graded pieces outside $[-n,0]$ allows us to use the finite filtration on $F_0K/F_{-n-1}K$. In particular, local identifications of these sheaves with torsion-free sheaves glue if they agree on a dense open in each overlap. Under these identifications, equality of $d_1$ with the exterior differential can likewise be checked on that open.

\begin{lemma}\label{lem:strictness}
Let $X$ be proper and $\mathcal M\in D^b\MHM(X)$. Then $R\pi_{X,*}\tDR_X(\mathcal M)$ is strict.
\end{lemma}
\begin{proof}
Let $a_X\colon X\to\mathrm{pt}$ be the structure map. Represent $Ra_{X,*}\mathcal M$ by a bounded complex of mixed Hodge structures. By proper direct image, the Rees complex of this bounded complex computes $R\pi_{X,*}\tDR_X(\mathcal M)$. The differentials are strict for $F$, so the degree-$j$ cohomology is the Rees module of $H^j(Ra_{X,*}\mathcal M)$ and has no $z$-torsion.
\end{proof}

\subsection{Smooth and simplicial comparisons}\label{sec:simplicial}

We use the open immersions of Section~\ref{sec:intro}. The filtered logarithmic comparisons for $[*D]$ and $[!D]$ are given in~\cite[Theorem~1]{Wei20}. The mixed-boundary comparison specializes to the following statement for the constant Hodge module.

\begin{theorem}\label{pre:smooth-comparison}
Let $X$ be a smooth variety of dimension $n$ and assume that $B+C$ is a simple normal crossing divisor. Then there is a canonical exchange isomorphism
\[
 j_{B,!}i_{C,*}\C_U^H[n]\simeq j_{C,*}i_{B,!}\C_U^H[n]
\]
in $D^b\MHM(X)$. These objects are mixed Hodge modules. If $\mathcal M$ denotes
either of them, then
\[
 \tDR_X(\mathcal M)\simeq\tDR_{(X,!B,*C)}^\bullet.
\]
\end{theorem}

\begin{proof}
The open immersions are affine because their complements are effective Cartier divisors. Their $!$- and $*$-extensions are therefore perverse $t$-exact, so the indicated extensions of $\C_U^H[n]$ are mixed Hodge modules. The constant Hodge module on $X$ is strictly non-characteristic along every stratum of $D=B+C$. Proposition~3.14 of~\cite{Wei23b} thus identifies the two extension orders and computes their Rees de Rham complex. In particular, the realization of either extension has zero stalks on $B$, so the exchange counit of Section~\ref{sec:intro} is an isomorphism. In right-module notation,~\cite[(3.36)]{Wei23b} gives the lattice $\omega_X(C)$ in generating degree $-n$. Tensoring with $\omega_X(D)^{-1}$ gives $\cO_X(-B)$, and the logarithmic $p$-forms have generating degree $-p$ as in~\eqref{pre:filtration}.
\end{proof}

Consider any finite toric quotient $f\colon\widehat X\to X$ obtained by replacing $N$ by a finite-index sublattice with the same fan. Let $G$ be its deck group and $\widehat B,\widehat C$ the reduced inverse images of $B,C$. Each fiber is a single $G$-orbit, so the adjunction $\C_X^H\to(Rf_*\C_{\widehat X}^H)^G$ is an isomorphism on every stalk. The same holds over each boundary complement. Proper base change for $!$-extension and composition of direct images for $*$-extension identify the exchange map on $X$ with the $G$-invariant direct image of the exchange map on $\widehat X$.

In~\eqref{pre:weight-formula}, taking $G$-invariants retains exactly the characters in $M$. The ray generators in the two lattices differ by positive factors, so their coefficient spaces agree. We obtain
\begin{equation}\label{sim:invariant-forms}
 \Om{p}{X}{B}{C}
 \simeq\bigl(f_*\Om{p}{\widehat X}{\widehat B}{\widehat C}\bigr)^G,
\end{equation}
compatibly with the differential $m\wedge(-)$ on each character summand.

\begin{proposition}\label{prop:simplicial-comparison}
Let $(X,!B,*C)$ be a simplicial toroidal triple of dimension $n$. The canonical exchange
morphism is an isomorphism
\[
 j_{B,!}i_{C,*}\C_U^H[n]\simeq j_{C,*}i_{B,!}\C_U^H[n].
\]
Both objects are mixed Hodge modules. For their common value $\mathcal M$, one has
\[
 \tDR_X(\mathcal M)\simeq\tDR_{(X,!B,*C)}^\bullet.
\]
In particular,
the complex realization of $\mathcal M$ is the perverse sheaf $\C_X[!B+*C][n]$.
If $X$ is proper, then $R\pi_{X,*}\tDR_{(X,!B,*C)}^\bullet$ is strict.
\end{proposition}

\begin{proof}
Exchange and membership in $\MHM(X)$ can be tested on common \'etale charts, since constructible realization detects both properties~\cite[Theorem~0.1 and Section~4.2]{Sai90}. On a simplicial chart, the primitive ray generators extend to a basis of a finite-index sublattice. The corresponding cover is smooth. Theorem~\ref{pre:smooth-comparison} and the finite-quotient calculation above give the exchange isomorphism. Its source is a mixed Hodge module by exactness of finite direct image and invariants. Proper direct image of the smooth Rees comparison, followed by invariants and~\eqref{sim:invariant-forms}, gives the comparison on the chart.

On a common \'etale refinement of two charts, both graded identifications are the ordinary de Rham comparison on the dense intersection of the inverse images of their tori. The targets $\Om{p}{X}{B}{C}$ are torsion-free, so the comparisons glue by the reconstruction in Section~\ref{sec:saito}. Proper strictness follows from Lemma~\ref{lem:strictness}.
\end{proof}

\section{Toroidal modifications}\label{sec:fans}

We first construct a small modification of $X$ on which the strict transform of $C-B$ is relatively ample. On each toric chart of this model, we then subdivide at the marked rays to obtain a log-simplicial triple.

\subsection{Sorting functions and convex subdivisions}
We need relatively ample divisors whose coefficients have the form $\mathbf{c}C-\mathbf{b}B$ with $\mathbf{b},\mathbf{c}\geq0$. Since divisor coefficients are the negatives of support-function values, this leads to the following signs.

\begin{definition}\label{fan:sorting-function}
A sorting function on the affine fan of a cone $\sigma$ is a rational
linear function $\rho$ on $\Span(\sigma)$ such that
\[
 \rho|_\nu\geq0\quad(\nu\in \mathfrak{B}_{\sigma}),\qquad
 \rho|_\nu\leq0\quad(\nu\in \mathfrak{C}_{\sigma}),\qquad
 \rho|_\nu=0\quad(\nu\in \mathfrak{A}_{\sigma}).
\]
On a fan, a sorting function is rational, continuous, and linear on
each cone, with these signs. It is
$\mathfrak{C}$-strict if it is negative on the primitive generator of each $\mathfrak C$-ray in its domain.
\end{definition}

A subdivision $\mathfrak{X}'\to\mathfrak{X}$ is a fan $\mathfrak{X}'$ in $N_\R$ with
$|\mathfrak{X}'|=|\mathfrak{X}|$ such that every cone of $\mathfrak{X}'$ is contained in
a cone of $\mathfrak{X}$. It is \emph{efficient} if
$\mathfrak{X}'(1)=\mathfrak{X}(1)$. For an efficient subdivision, we use the
inherited markings $\mathfrak{B}'=\mathfrak{B}$ and $\mathfrak{C}'=\mathfrak{C}$.
The corresponding toric birational morphism is small, since every prime divisor meets the isomorphism locus.

For $\sigma\in\mathfrak{X}$, put
$\mathfrak{X}'_\sigma=\{\tau\in\mathfrak{X}':\tau\subseteq\sigma\}$.
A \emph{good function} for this subdivision is a continuous rational
piecewise linear function $\psi_\sigma\colon\sigma\to\R$ such that,
for each maximal cone $\tau$ of $\mathfrak{X}'_\sigma$, there is
$m_\tau\in M_\Q$ with
\[
 \psi_\sigma(v)\leq\langle m_\tau,v\rangle\quad(v\in\sigma),
 \qquad
 \{v\in\sigma:\psi_\sigma(v)=\langle m_\tau,v\rangle\}=\tau.
\]
Thus $\psi_\sigma=\min_\tau\langle m_\tau,-\rangle$ is upper convex,
meaning $\psi_\sigma(v+w)\geq\psi_\sigma(v)+\psi_\sigma(w)$ for
$v,w\in\sigma$, and its maximal domains of linearity are exactly
the maximal cones of $\mathfrak{X}'_\sigma$.
A sufficiently divisible positive integral multiple of a good function is the support function of a relatively ample invariant Cartier divisor over $U_\sigma$~\cite[Theorem~7.2.6]{CLS11}.

\begin{definition}\label{def:convex}
Let $(\mathfrak{X},!\mathfrak{B},*\mathfrak{C})$ be a fan triple, and let $\mathfrak{X}'\to\mathfrak{X}$ be
an efficient subdivision with the inherited markings $\mathfrak{B},\mathfrak{C}$.
We call this subdivision \emph{locally convex} if, for every
$\sigma\in\mathfrak{X}$, there is a good function $\psi_\sigma$ for
$\mathfrak{X}'_\sigma$ which is also a sorting function for
$(\mathfrak{X}'_\sigma,!\mathfrak{B}_{\sigma},*\mathfrak{C}_{\sigma})$.
We call such a $\psi_\sigma$ a \emph{good sorting function} on $\sigma$. The choices on different cones need not agree on their common faces.
\end{definition}

\subsection{The canonical divisorial model}

\begin{proposition}\label{prop:canonical-model}
Let $(X,!B,*C)$ be a toroidal triple and put $\Delta=C-B$. The divisorial algebra
\begin{equation}\label{fan:divisorial-algebra}
 \cR_X(\Delta)=\bigoplus_{m\geq0}\cO_X(m\Delta)
\end{equation}
is locally finitely generated. There is a canonical projective small
birational morphism
\begin{equation}\label{fan:canonical-proj}
 f\colon Y=\Proj_X\cR_X(\Delta)\longrightarrow X.
\end{equation}
Let $B_Y,C_Y$ be the strict transforms of $B,C$. Then $\Delta_Y=C_Y-B_Y$ is $\Q$-Cartier and $f$-ample.
On each toric chart of $Y$, let $\mathfrak B_Y$ and $\mathfrak C_Y$ denote the ray sets of $B_Y$ and $C_Y$. The support function of $\Delta_Y$ takes the values $1$, $-1$, and $0$ on the primitive generators of $\mathfrak B_Y$-rays, $\mathfrak C_Y$-rays, and unmarked rays, respectively.
In particular, its restriction to each cone is a $\mathfrak{C}_Y$-strict sorting
function. The morphism is an isomorphism over the locus where $\Delta$
is $\Q$-Cartier.

On each toric chart, $f$ is induced by an efficient locally convex subdivision. These subdivisions commute with restriction to faces and with lattice isomorphisms preserving the markings.
\end{proposition}
\begin{proof}
We first compute the algebra and its fan on the affine toric variety of a
full-dimensional strongly convex cone
$\sigma=\Cone(v_1,\ldots,v_r)$, with primitive ray generators $v_i$.
Set
\[
 h_i=\begin{cases}
 1,&\R_{\geq0}v_i\in \mathfrak{B},\\
 -1,&\R_{\geq0}v_i\in \mathfrak{C},\\
 0,&\text{otherwise},
 \end{cases}
 \qquad
 P_\sigma=\{u\in M_\R:\langle u,v_i\rangle\geq h_i
                  \text{ for all }i\}.
\]
The graded algebra on this chart is
\begin{equation}\label{fan:semigroup}
 \C\bigl[\{(u,m)\in M\times\Z_{\geq0}:
                  \langle u,v_i\rangle\geq m h_i\text{ for all }i\}\bigr].
\end{equation}
The semigroup in~\eqref{fan:semigroup} consists of the lattice points of a rational polyhedral cone, so Gordan's lemma gives finite generation~\cite[Proposition~1.2.17]{CLS11}.

For $v\in\sigma$, put
\begin{equation}\label{fan:height-envelope}
 \varphi_\sigma(v)=\min_{u\in P_\sigma}\langle u,v\rangle.
\end{equation}
A large multiple of an interior point of $\sigma^\vee$ satisfies all the inequalities strictly. The recession cone $\sigma^\vee$ is pointed and $v\in\sigma$ is nonnegative on it, so the minimum is attained at a vertex. The maximal domains of linearity of $\varphi_\sigma$ are the normal cones
\begin{equation}\label{fan:normal-cones}
 \Cone\{v_i:\langle u,v_i\rangle=h_i\},
 \qquad u\text{ a vertex of }P_\sigma.
\end{equation}
The normal fan of $P_\sigma$ is the fan of~\eqref{fan:canonical-proj}. Indeed, for any positive integer $k$ clearing the vertex denominators, the cone algebra of $kP_\sigma$ is the $k$th Veronese of~\eqref{fan:semigroup}, so~\cite[Theorem~7.1.13]{CLS11} applies.

Every extremal ray of $\sigma$ persists in the subdivision, and~\eqref{fan:normal-cones} introduces no new rays. Hence the subdivision is efficient. For a face $\tau\preceq\sigma$, let $P_\tau$ be given by the inequalities on $\tau$ in $(\Span_\R\tau)^*$. Every $u\in P_\tau$ extends to $P_\sigma$ by adding to any linear extension a sufficiently large multiple of a supporting form vanishing precisely on $\tau$ in $\sigma$. Taking minima, first on a ray and then on an arbitrary face, gives
\begin{align}
 \varphi_\sigma(v_i)&=h_i,\label{fan:ray-heights}\\
 \varphi_\sigma|_\tau&=\varphi_\tau.\label{fan:face-compatibility}
\end{align}
Thus $\varphi_\sigma$ is the support function of $\Delta_Y$. Its ray values give the sorting signs, and its domains of linearity give relative ampleness. The same calculation in the saturated span treats a cone of lower dimension.

Divisorial sheaves and their multiplication commute with \'etale pullback. Hence finite generation descends from the toric charts, and relative Proj gives a projective morphism with the computed local fans. In particular, $Y$ is normal, $f$ is small, and the construction respects isomorphisms of marked charts. On the open set where $r\Delta$ is Cartier, the $r$th Veronese algebra is the symmetric algebra of $\cO_X(r\Delta)$. Its relative Proj is that open set.
\end{proof}

\begin{example}\label{ex:hexagon}
In $N=\Z^3$, let $\sigma$ be the cone with primitive ray generators
\[
\begin{array}{lll}
 v_1=(1,0,1),&v_2=(1,1,1),&v_3=(0,1,1),\\
 v_4=(-1,0,1),&v_5=(-1,-1,1),&v_6=(0,-1,1).
\end{array}
\]
Mark the rays of $v_1,v_3$ for $B$ and those of $v_4,v_6$ for $C$. A sorting function $\rho$ on $\sigma$ must vanish on the unmarked generators $v_2,v_5$, so $\rho(x,y,z)=a(x-y)$. Nonnegativity on $v_1,v_3$ then forces $a=0$. Thus $\sigma$ admits no $\mathfrak C$-strict sorting function.

For the canonical subdivision, the prescribed heights are $(1,0,1,-1,0,-1)$. Solving $\langle u,v_i\rangle=h_i$ for $i=1,3,5$ gives $u=(1/3,1/3,2/3)$. This point satisfies the other three inequalities strictly, so~\eqref{fan:normal-cones} gives $\Cone(v_1,v_3,v_5)$ as a maximal cone.

\begin{center}
\begin{tikzpicture}[scale=1.45]
 \coordinate (v1) at (1,0);
 \coordinate (v2) at (1,1);
 \coordinate (v3) at (0,1);
 \coordinate (v4) at (-1,0);
 \coordinate (v5) at (-1,-1);
 \coordinate (v6) at (0,-1);
 \fill[gray!15] (v1)--(v3)--(v5)--cycle;
 \draw (v1)--(v2)--(v3)--(v4)--(v5)--(v6)--cycle;
 \draw (v1)--(v3)--(v5)--cycle;
 \fill (v1) circle (1.7pt) (v3) circle (1.7pt);
 \draw[fill=white] (v4) circle (1.7pt) (v6) circle (1.7pt);
 \fill[gray] (v2) circle (1.7pt) (v5) circle (1.7pt);
 \node[right] at (v1) {$v_1\ (\mathfrak{B})$};
 \node[above right] at (v2) {$v_2$};
 \node[above] at (v3) {$v_3\ (\mathfrak{B})$};
 \node[left] at (v4) {$v_4\ (\mathfrak{C})$};
 \node[below left] at (v5) {$v_5$};
 \node[below] at (v6) {$v_6\ (\mathfrak{C})$};
\end{tikzpicture}
\end{center}
In the section $z=1$, the shaded triangle corresponds to $\Cone(v_1,v_3,v_5)$. The three complementary triangles give the remaining maximal cones, since the canonical subdivision introduces no new rays.
\end{example}

\subsection{Log-simplicial models}\label{fan:log-simplicial-models}
We now refine the local fans of the canonical model so that the marked rays are independent factors.

\begin{definition}\label{def:log-simplicial}
Let $\mathfrak{E}$ be a set of rays of $\mathfrak{X}$. We say that $\mathfrak{X}$ is
$\mathfrak{E}$-simplicial if, for each cone $\sigma$ and each $\nu\in\mathfrak{E}_{\sigma}$, there is a face $\zeta\preceq\sigma$ with $\sigma=\Cone(\zeta,\nu)$ and $\dim\sigma=\dim\zeta+1$. Equivalently, each cone has a decomposition
\[
 \sigma=\tau+\sum_{\nu\in \mathfrak{E}_{\sigma}}\nu,
 \qquad
 \dim\sigma=\dim\tau+|\mathfrak{E}_{\sigma}|,
\]
where $\tau$ is the face spanned by the remaining rays.
A toric triple is \emph{log-simplicial} if its fan is $(\mathfrak B\cup\mathfrak C)$-simplicial. For a toroidal triple, we require this condition in its local toric charts.
\end{definition}

\begin{lemma}\label{lem:join}
Let $\sigma$ have the decomposition of Definition~\ref{def:log-simplicial}. Write $v_1,\ldots,v_s$ for the primitive generators of the rays in $\mathfrak E_\sigma$ and $\tau$ for the face spanned by the remaining rays.
\begin{enumerate}[label=\textup{(\arabic*)}]
\item Every efficient subdivision of $\sigma$ is obtained from an
efficient subdivision of $\tau$ by taking independent joins with the
faces of $\sum_i\R_{\geq0}v_i$.
\item There is a finite toric quotient $S_\tau\times\mathbb A^s\to S_\sigma$ by a finite abelian group. For each $i$, the $i$th coordinate hyperplane of the $\mathbb A^s$-factor maps onto the divisor corresponding to $\R_{\geq0}v_i$.
\end{enumerate}
\end{lemma}
\begin{proof}
For (1), every maximal cone of the subdivision has dimension $\dim\sigma$. Its rays outside $\tau$ must therefore include all $\R_{\geq0}v_i$, since their images modulo $\Span(\tau)$ are independent and the subdivision is efficient. Each maximal cone is the join of its intersection with $\tau$ and these rays. Taking faces gives the assertion.

For (2), the direct sum $N'=(N\cap\Span_\R(\tau))+\Z v_1+\cdots+\Z v_s$ has finite index in $N\cap\Span_\R(\sigma)$ and makes the cone a product. The inclusion of lattices gives the finite quotient map, and the ray of $v_i$ defines the $i$th coordinate hyperplane on the cover. For the full affine chart $U_\sigma$, take the product with the complementary torus factor.
\end{proof}

For a log-simplicial triple, take $\mathfrak E=\mathfrak B\cup\mathfrak C$. Since the marked rays are independent factors, assigning $-1$ to the primitive $\mathfrak C$-generators and zero to all other rays defines a $\mathfrak C$-strict sorting function that is linear on each cone. If $\nu$ is an unmarked ray, quotienting a cone containing $\nu$ by $\Span_\R(\nu)$ replaces only its unmarked factor by its image. The induced triple on the divisor of $\nu$ is therefore log-simplicial.

In the star subdivision at an existing ray $\nu$, each cone containing $\nu$ is replaced by the joins $\Cone(\nu,\tau)$ with its faces $\tau$ avoiding $\nu$, together with their faces. The subdivision restricts to the identity on cones avoiding $\nu$. The subdivision is efficient and makes $\nu$ an independent factor in every cone containing it.

\begin{proposition}\label{cor:ordered-stars}
Suppose that every cone of $(\mathfrak{X},!\mathfrak{B},*\mathfrak{C})$ admits a $\mathfrak{C}$-strict sorting function. Given any order $\nu_1\prec\cdots\prec\nu_m$ on $\mathfrak{B}\cup\mathfrak{C}$, take the star subdivisions successively at $\nu_m,\ldots,\nu_1$. Each step is efficient and locally convex, and the final fan is log-simplicial. The construction commutes with restriction to faces with their induced orders.
\end{proposition}
\begin{proof}
At a center $\nu$, let $\psi_\nu$ take value one at the primitive generator $v_\nu$ and zero at the other primitive generators. It extends linearly on each join $\Cone(\nu,\tau)$ by vanishing on $\tau$ and is zero on unaffected cones. Fix an old cone $\sigma$ containing $\nu$. For every facet $\tau$ of $\sigma$ avoiding $\nu$, take a rational supporting form $\ell_\tau\geq0$ vanishing precisely on $\tau$ in $\sigma$, normalized by $\ell_\tau(v_\nu)=1$. On the rays, $\ell_\tau\geq\psi_\nu$, with equality precisely on $\nu$ and the rays of $\tau$. By linearity on each new cone, the inequality holds on $\sigma$ with equality exactly on $\Cone(\nu,\tau)$. Thus $\psi_\nu$ is good over $\sigma$ and has the sorting signs if $\nu\in\mathfrak B$.

If $\nu\in\mathfrak C$, the hypothesis supplies a $\mathfrak C$-strict sorting function $\rho$ on an original cone containing $\sigma$. Its restriction is linear on $\sigma$. For $a=-1/\rho(v_\nu)>0$, the function $\psi_\nu+a\rho$ vanishes on $\nu$, has the required signs on all other rays, and has the same domains of linearity as $\psi_\nu$. This proves local convexity.

Each star makes its center an independent ray factor, while
Lemma~\ref{lem:join} shows that subsequent efficient subdivisions
preserve the previous factors. Hence the final fan is log-simplicial.
On restricting to a face, the stars at rays outside the face act trivially and the other centers keep their order.
\end{proof}

Proposition~\ref{prop:canonical-model} supplies the $\mathfrak C$-strict sorting functions required in Proposition~\ref{cor:ordered-stars}. On a toroidal triple, we perform the ordered stars on each chart of the canonical modification, since monodromy may permute the marked rays.

\section{Relative Bott vanishing}\label{sec:bott}

Following~\cite[\S2]{Wei23a}, we prove relative Bott vanishing by induction on the number of unmarked rays in the source fan. To restrict to an unmarked invariant divisor, we first determine how the Weil twist changes its boundary.

\subsection{Twisted restriction and residue}

Let $i\colon E\hookrightarrow X$ be an unmarked invariant prime divisor of a toric triple $(X,!B,*C)$, with ray generator $e$, and let $L$ be a $\Q$-Cartier Weil divisor on $X$. The quotient of the natural inclusion $\cO_X(L-E)\subset\cO_X(L)$ is the direct image of a rank-one reflexive sheaf on $E$. Thus the exact sequence
\[
 0\longrightarrow\cO_X(L-E)\longrightarrow\cO_X(L)
 \longrightarrow i_*\cO_E(L_E)\longrightarrow0
\]
determines $L_E$ up to linear equivalence.

To compute the quotient, represent $L$ by an invariant divisor with coefficient zero along $E$. Each invariant prime $D_\rho^E$ on $E$ corresponds to a two-dimensional face $\Cone(e,v_\rho)$. Write $\overline v_\rho=d_\rho w_\rho$ in $N/\Z e$, with $w_\rho$ primitive, and let $l_\rho$ be the coefficient of $L$ at $D_\rho$. On a cone containing $e$, the linear form giving the coefficients of $L$ vanishes on $e$. It therefore descends to the quotient cone, where its value at $w_\rho$ is $l_\rho/d_\rho$, the coefficient of $L|_E$ at $D_\rho^E$. The quotient sheaf retains precisely the characters $m\in M\cap e^\perp$ satisfying $\langle m,w_\rho\rangle+l_\rho/d_\rho\geq0$ on every ray of the quotient cone. Since $\langle m,w_\rho\rangle$ is integral, these inequalities give the representative $L_E=\lfloor L|_E\rfloor$.

To define $B^E,C^E$, retain the marking of $D_\rho$ when $l_\rho/d_\rho$ is integral and assign $*D_\rho^E$ otherwise. A change of normalized representative changes $L_E$ by a principal divisor and preserves the markings. With these markings and the twist $L_E$ on $E$, the cone conditions of Proposition~\ref{prop:restriction} still suffice for exactness of the twisted restriction and residue sequences.

\begin{proposition}\label{lem:invariant-twisted-restriction}
Assume the hypotheses for restriction in Proposition~\ref{prop:restriction}. For $p\geq0$, restriction gives an exact sequence
\begin{subequations}
\begin{equation}\label{proof:invariant-twisted-restriction}
\begin{aligned}
0&\longrightarrow\Tw{p}{X}{(B+E)}{C}{L}
 \longrightarrow\Tw{p}{X}{B}{C}{L}\\
 &\hspace{18mm}\xrightarrow{\operatorname{restr}_E}i_*\Tw{p}{E}{B^E}{C^E}{L_E}\longrightarrow0.
\end{aligned}
\end{equation}
Under the additional face hypothesis for residue in Proposition~\ref{prop:restriction}, residue gives an exact sequence
\begin{equation}\label{proof:twisted-residue-coordinate}
\begin{aligned}
0&\longrightarrow\Tw{p}{X}{B}{C}{L}
 \longrightarrow\Tw{p}{X}{B}{(C+E)}{L}\\
 &\hspace{18mm}\xrightarrow{\operatorname{Res}_E}i_*\Tw{p-1}{E}{B^E}{C^E}{L_E}\longrightarrow0,
\end{aligned}
\end{equation}
\end{subequations}
with zero residue term for $p=0$.

Both sequences also hold along $Z\times E\subset Z\times X$ with the markings from each factor and coefficient $\mathcal A\boxtimes\cO_X(L)$, where $Z$ is normal with logarithmic forms as in~\eqref{pre:intrinsic-forms} and $\mathcal A$ is rank-one reflexive. On $Z\times E$, one retains the markings on $Z$ and takes $B^E,C^E$ on $E$, with coefficient $\mathcal A\boxtimes\cO_E(L_E)$.

If $L\sim_{\Q}A+\mathbf bB-\mathbf cC$ on $X$, with $A$ $\Q$-Cartier and $0\leq b_i,c_j\leq1$, then $L_E\sim_{\Q}A|_E+\mathbf b^E B^E-\mathbf c^E C^E$ with $0\leq b_i^E,c_j^E\leq1$.
\end{proposition}

\begin{proof}
On a cone containing $e$, choose $m_L\in M_\Q$ with $\langle m_L,v_\rho\rangle=l_\rho$ on its rays and $k\in\Z_{>0}$ with $km_L\in M$. In degree $m\in M$, the twisted orders are $\langle m+m_L,v_\rho\rangle$. The untwisted orders in degree $k(m+m_L)$ are their $k$-fold multiples, so~\eqref{pre:weight-formula} gives the same coefficient spaces on $X$.

Only $\langle m,e\rangle=0$ contributes to the target, since $\langle m_L,e\rangle=0$. Write $l_\rho=d_\rho q_\rho+s_\rho$ with $0\leq s_\rho<d_\rho$. On $E$, the untwisted order in the shifted degree, divided by $k$, is
\[
 \langle m+m_L,w_\rho\rangle
 =\langle m,w_\rho\rangle+q_\rho+\frac{s_\rho}{d_\rho}.
\]
For $s_\rho=0$, this is the $L_E$-twisted order with the inherited marking. For $s_\rho>0$, it never vanishes and is positive exactly when $\langle m,w_\rho\rangle+q_\rho\geq0$. This gives twist $q_\rho$ and allows logarithmic poles along $D_\rho^E$, as prescribed by $C^E$. In normal order zero, restriction is inclusion on coefficient spaces and residue is contraction by $e$. These are the maps of Proposition~\ref{prop:restriction} in degree $k(m+m_L)$, giving exactness under the stated hypotheses. In any other normal order, the target vanishes and the two source spaces agree. The maps agree generically on $E$ with restriction and residue, so they glue by torsion-freeness.

On $Z\times X$, twisted $p$-forms are the direct sum of the exterior products of $\mathcal A$-twisted $a$-forms on $Z$ and $L$-twisted $b$-forms on $X$ for $a+b=p$. This follows on the smooth normal crossing loci with invertible coefficients and extends by reflexivity. Exterior product with each sheaf on $Z$ preserves exactness over $\C$. Residue on the summand of $Z$-degree $a$ has sign $(-1)^a$.

For the bounds, represent $A$ by $L-\mathbf bB+\mathbf cC$ up to $\Q$-linear equivalence, and let $t_\rho$ be the coefficient of $\mathbf bB-\mathbf cC$ at $D_\rho$. The coefficient of $L_E-A|_E$ is $q_\rho-(l_\rho-t_\rho)/d_\rho=(t_\rho-s_\rho)/d_\rho$. It lies in $[0,1]$ on $B^E$, in $[-1,0]$ on $C^E$, and is zero elsewhere.
\end{proof}

On a log-simplicial triple, the unmarked rays of each cone generate a face. Every marked ray in a cone containing $e$ spans a two-dimensional face with $e$, so both boundary traces are divisors or empty. Thus~\eqref{proof:invariant-twisted-restriction} holds for every unmarked invariant prime divisor $E$ and recovers~\eqref{pre:restriction-sequence} when $L=0$.

\subsection{Proof of relative Bott vanishing}\label{proof:toric-bott}

Fujino proved the case of a relatively ample line bundle by a splitting under toric multiplication and Serre vanishing~\cite[Proposition~3.2 and Theorem~4.3]{Fuj08}. For a Weil divisor with rational boundary coefficients, we start with relative Kawamata--Viehweg vanishing on a plenary log-simplicial triple.

\begin{theorem}\label{thm:simplicial-bott}
Let $f\colon X\to S$ be a projective toric morphism with $S$ affine, and let
$(X,!B,*C)$ be a toric triple. Suppose that a Weil divisor $L$ satisfies
\[
 L\sim_{\Q} A+\mathbf{b}B-\mathbf{c}C,
\]
where $A$ is an $f$-ample $\Q$-Cartier divisor and all entries of $\mathbf{b},\mathbf{c}$
belong to $[0,1]$. Then
\[
 R^qf_*\Tw{p}{X}{B}{C}{L}=0
 \qquad(q>0,\ 0\leq p\leq\dim X).
\]
\end{theorem}

\begin{proof}
We induct on the number $a$ of unmarked rays of $X$. For each $a$, we first prove the log-simplicial case and then the general case.

Suppose first that the triple is log-simplicial. If $a=0$, the triple is
plenary and $X$ is simplicial. Thus
\[
 \Tw{p}{X}{B}{C}{L}
 \simeq\bigwedge^p M_\C\otimes\cO_X(L-B),
 \qquad
 L-B\sim_{\Q}K_X+A+\mathbf bB+(\mathbf1-\mathbf c)C.
\]
The invariant boundary $\mathbf bB+(\mathbf1-\mathbf c)C$ has coefficients in $[0,1]$, so the toric pair is log canonical. The Weil divisor $L-B$ is $\Q$-Cartier, and $A$ is relatively ample on $X$ and on each log canonical center. Thus~\cite[Theorem~5.7.6]{Fuj17} gives the required higher direct-image vanishing.

If $a>0$, choose an unmarked divisor $E$ with ray generator $e$. Since the marked rays are independent factors on each cone, the marked divisors are $\Q$-Cartier. The given relation then makes $L$ $\Q$-Cartier, so Proposition~\ref{lem:invariant-twisted-restriction} gives the restriction sequence. In the first term of~\eqref{proof:invariant-twisted-restriction}, adding $E$ to $B$ with coefficient zero leaves $a-1$ unmarked rays. In the last term, each unmarked ray of $E$ corresponds to a distinct two-dimensional cone $\Cone(e,v_\rho)$ with $\rho$ unmarked on $X$ and different from the ray of $E$. Thus $E$ has at most $a-1$ unmarked rays.

Let $O(\tau)$ be the orbit of $S$ containing the image of the dense orbit of $E$. The restriction of $f$ factors through a projective toric morphism $E\to V(\tau)=\overline{O(\tau)}$ with $V(\tau)$ affine and $A|_E$ relatively ample. The general induction hypothesis for fewer than $a$ unmarked rays gives vanishing for the last term over $V(\tau)$ when $p\leq\dim E$, and the term is zero for $p>\dim E$. The closed immersion $V(\tau)\hookrightarrow S$ has exact direct image. Induction also gives vanishing for the first term over $S$, so the long exact sequence gives vanishing for the middle term.

For an arbitrary triple, put $T=\mathbf bB-\mathbf cC$. Choose $T'=\mathbf b'B-\mathbf c'C$ with rational coefficients $0<b'_i,c'_j<1$, as follows. For each minimally dependent set of primitive ray generators lying in one cone and containing a marked ray, fix a nonzero relation $\sum_\rho a_\rho v_\rho=0$. Write $h_\rho$ for the coefficient of $H=T-T'$ at $D_\rho$. We require $\sum_\rho a_\rho h_\rho\ne0$ for each chosen relation. The marked coefficients of $T'$ vary independently, so each relation excludes a proper hyperplane. A rational choice avoiding all of them exists.

Apply the divisorial construction in the proof of Proposition~\ref{prop:canonical-model} to an integral multiple of $H$. The same polyhedron with ray heights $-h_\rho$ gives a projective small modification $g\colon Z\to X$ on which the strict transform $H_Z$ is $g$-ample and $\Q$-Cartier. If a new cone contained a minimally dependent set with a marked ray, linearity of its heights would give $\sum_\rho a_\rho h_\rho=0$, contrary to the choice of $T'$. Thus each marked generator lies outside the span of the others. A linear form taking value one on that generator and zero on the others cuts out the opposite facet, as required for log-simpliciality.

For sufficiently small rational $0<\epsilon<1$, $g^*A+\epsilon H_Z$ is $(f\circ g)$-ample. Taking strict transforms under the small modification $g$ gives
\[
 L_Z\sim_\Q(g^*A+\epsilon H_Z)+(1-\epsilon)T_Z+\epsilon T'_Z,
\]
where subscripts denote strict transforms. The boundary coefficients satisfy the required bounds, including when an original coefficient was $0$ or $1$. The log-simplicial case applies to $f\circ g$, since no rays were added. It also applies to $g$ over each affine toric chart of $X$, where the ample part is $g$-ample and the inverse image has at most $a$ unmarked rays. Put $\mathcal F=\Tw{p}{Z}{B_Z}{C_Z}{L_Z}$. Proposition~\ref{prop:pushforward} gives $g_*\mathcal F=\Tw{p}{X}{B}{C}{L}$. Since $R^jg_*\mathcal F=0$ for $j>0$, Leray gives $R^qf_*g_*\mathcal F=R^q(f\circ g)_*\mathcal F=0$ for $q>0$.
\end{proof}

We apply relative Bott vanishing to the logarithmic Rees complex of a locally convex modification.

\begin{lemma}\label{lem:relative-vanishing}
Let $f\colon Y\to X$ be an efficient locally convex modification of a toric triple $(X,!B,*C)$, and let $(Y,!B_Y,*C_Y)$ be the induced triple. Then
\[
 Rf_*\Om{p}{Y}{B_Y}{C_Y}\simeq\Om{p}{X}{B}{C}
\]
for every $p$. The same conclusion holds for a composition of such modifications and, locally on the base, for toroidal modifications with these toric models.
Moreover, with $\widetilde f=f\times\id_{\mathbb A^1}$, pullback of forms gives a canonical
Rees comparison
\begin{equation}\label{proof:relative-rees}
 \tDR_{(X,!B,*C)}^\bullet
 \xrightarrow{\ \sim\ }R\widetilde f_*\tDR_{(Y,!B_Y,*C_Y)}^\bullet.
\end{equation}
\end{lemma}

\begin{proof}
On an affine invariant chart of $X$, local convexity supplies a good sorting function. Rescaled by a sufficiently small positive rational number, it gives an $f$-ample $\Q$-Cartier divisor $A_f=\mathbf c_fC_Y-\mathbf b_fB_Y$ with all entries of $\mathbf b_f,\mathbf c_f$ in $[0,1]$. Applying Theorem~\ref{thm:simplicial-bott} to $0=A_f+\mathbf b_fB_Y-\mathbf c_fC_Y$ gives $R^qf_*\Om{p}{Y}{B_Y}{C_Y}=0$ for $q>0$.

Proposition~\ref{prop:pushforward} identifies the degree-zero direct image by pullback on the common isomorphism locus. Leray gives the result for compositions, and flat base change on common \'etale charts gives it for toroidal modifications by faithfully flat descent.

Proper coherent direct-image comparison gives the same direct images for the analytic sheaves of logarithmic forms. In each internal degree, the Rees term $\Om{p}{Y}{B_Y}{C_Y}z^{-p}\C[z]$ is a copy of $\Om{p}{Y}{B_Y}{C_Y}$ or zero. Its higher direct images therefore vanish degree by degree, and pullback identifies its direct image with the corresponding term on $X$. These identifications respect multiplication by $z$ and commute with $d$ by Proposition~\ref{prop:pushforward}. Thus the direct images of the terms, with differential $z\,d$, compute~\eqref{proof:relative-rees}.
\end{proof}

\section{Hodge realizations and Kawamata--Viehweg vanishing}\label{sec:proofs}

\subsection{Untwisted Hodge realizations}
\label{proof:untwisted-comparisons}

We use the boundary operations of Section~\ref{pre:twistor-dr} in $D^b\MHM(X)$. To realize the logarithmic complex on $X$, we first identify the direct image of ordered boundary extensions under the star subdivisions of Proposition~\ref{cor:ordered-stars}.

\begin{lemma}\label{lem:ordered-descent}
Let $D_1,\ldots,D_s$ be distinct invariant prime divisors on a toric variety $X$, each assigned an extension symbol $!$ or $*$. Let $f\colon Y\to X$ be the toric morphism obtained by subdividing successively at their rays in reverse order, and write $D_i^Y$ for the strict transforms. The subdivisions fix the cones avoiding all these rays, so $f$ is an isomorphism over $U=X\setminus(D_1+\cdots+D_s)$. Given $\mathcal N\in D^b\MHM(U)$, let $\mathcal N_X$ and $\mathcal N_Y$ be its successive extensions across the $D_i$ and $D_i^Y$, respectively, in the given order and with the assigned symbols. There is a natural isomorphism
\[
 Rf_*\mathcal N_Y\simeq\mathcal N_X
\]
whose restriction to $U$ is the identity of $\mathcal N$.
\end{lemma}

\begin{proof}
We induct on $s$, with $s=0$ immediate. Let $\nu$ be the ray of $D_s$. After the first star, each cone containing $\nu$ is the join of $\nu$ with an original face avoiding $\nu$. By Lemma~\ref{lem:join}(1), later efficient stars refine only the complementary faces. Hence every final cone avoiding the ray lies in an original face avoiding it, giving
\begin{equation}\label{proof:star-boundary-equality}
 (f^{-1}D_s)_{\mathrm{red}}=D_s^Y.
\end{equation}
Write $j\colon X^\circ=X\setminus D_s\hookrightarrow X$ and $j'\colon Y^\circ=Y\setminus D_s^Y\hookrightarrow Y$. The restriction $f^\circ\colon Y^\circ\to X^\circ$ is given by the remaining stars. By induction, $Rf^\circ_*$ takes the first $s-1$ extensions on $Y^\circ$ to those on $X^\circ$. For a final $*$, composition gives $Rf_*Rj'_*\simeq Rj_*Rf^\circ_*$. For a final $!$, properness of $f$ and $f^\circ$ gives $Rf_*j'_!\simeq Rf_!j'_!\simeq j_!Rf^\circ_*$. Both identifications are natural and restrict to the identity on $U$.
\end{proof}

\begin{theorem}\label{thm:hodge-realization}
Let $(X,!B,*C)$ be a toroidal triple of pure dimension $n$, and let $f\colon Y\to X$ be the canonical modification of Proposition~\ref{prop:canonical-model}, with strict transforms $B_Y,C_Y$. Put $\mathcal M_X=Rf_*\bigl(\C_Y^H[*C_Y][!B_Y][n]\bigr)\in D^b\MHM(X)$. Then
\begin{equation}\label{proof:hodge-realization}
\tDR_X(\mathcal M_X)\simeq\tDR_{(X,!B,*C)}^\bullet.
\end{equation}
\end{theorem}

\begin{proof}
On charts admitting a $\mathfrak C$-strict sorting function, we shall use the ordered extension
\begin{equation}\label{proof:canonical-hodge-object}
 \mathcal N_X=\C_X^H[*C][!B][n].
\end{equation}
For a general triple, we apply the chartwise comparison on $Y$ and take its direct image.

\emph{Log-simplicial charts.} On every affine log-simplicial toric chart, we prove by induction on $n$ that
\begin{equation}\label{proof:canonical-graded}
 \bigl(\tDR_X(\mathcal N_X)|_{z=0}\bigr)_{-p}
 \simeq\Om{p}{X}{B}{C}[n-p]
\end{equation}
for all $p$, with zero right side for $p\notin[0,n]$. We construct the identifications to agree with ordinary de Rham comparison on the dense torus. This normalization determines each map to the torsion-free sheaf $\Om{p}{X}{B}{C}$, so the maps agree on overlaps and with smooth comparison on $X_{\reg}$. The assertion on a torus includes the initial case $n=0$.

If the cone has $s>0$ marked rays, Lemma~\ref{lem:join} gives a finite quotient $Z\times\mathbb A^s\to X$ with $\dim Z=n-s$. The coordinate hyperplanes lie over the marked divisors. In $D^b\MHM(Z\times\mathbb A^s)$, the ordered extension on the cover is the external product of $\C_Z^H[n-s]$ with one factor on each coordinate line. For $j\colon\C^*\hookrightarrow\mathbb A^1$, this factor is $j_!\C_{\C^*}^H[1]$ for a $B$-ray and $Rj_*\C_{\C^*}^H[1]$ for a $C$-ray. The induction hypothesis on $Z$ and Theorem~\ref{pre:smooth-comparison} give~\eqref{proof:canonical-graded} on the product, by compatibility of filtered de Rham with external products~\cite[Proposition~1.5, Remark~1.6(ii), and Section~1.8]{MSS11}. On the graded pieces, the external products of forms of total degree $p$ give the logarithmic $p$-forms by~\eqref{pre:weight-formula}. The deck transformations preserve the normalized comparisons. Taking invariants in the finite direct image, as before Proposition~\ref{prop:simplicial-comparison}, gives~\eqref{proof:canonical-graded} on $X$.

It remains to treat $B=C=0$ in dimension $n$. If the cone has a ray, choose one and let $i\colon E\hookrightarrow X$ be its invariant divisor and $g\colon Z\to X$ its star subdivision. On every maximal cone of the subdivision, the ray is an independent factor. The marked case just proved gives~\eqref{proof:canonical-graded} on $(Z,!E_Z,*0)$, where $E_Z$ is the strict transform of $E$. Lemma~\ref{lem:ordered-descent} gives $Rg_*\C_Z^H[!E_Z][n]\simeq\C_X^H[!E][n]$. The star is efficient and locally convex by Proposition~\ref{cor:ordered-stars}, so the proper direct-image formula for graded de Rham and Lemma~\ref{lem:relative-vanishing} give~\eqref{proof:canonical-graded} for $(X,!E,*0)$.

Use the localization triangle
\begin{equation}\label{proof:hodge-localization-triangle}
 \C_X^H[!E][n]\longrightarrow\C_X^H[n]\longrightarrow i_*\C_E^H[n]\longrightarrow\C_X^H[!E][n+1].
\end{equation}
Take the $(-p)$-th graded piece of the filtered de Rham complex of each term in~\eqref{proof:hodge-localization-triangle}. The marked case gives $\Om{p}{X}{E}{0}[n-p]$ for the first term. For the third, write $\C_E^H[n]=\C_E^H[n-1][1]$. Induction on $\dim E=n-1$ and closed direct image give $i_*\Om{p}{E}{0}{0}[n-p]$. Hence all three terms have cohomology only in degree $p-n$. The middle cohomology sheaf $T^p$ in that degree fits into the exact sequence
\begin{equation}\label{proof:hodge-graded-extension}
 0\longrightarrow\Om{p}{X}{E}{0}\longrightarrow T^p\longrightarrow i_*\Om{p}{E}{0}{0}\longrightarrow0.
\end{equation}
Let $V$ be the simplicial locus of $X$. Every cone of dimension at most two is simplicial, so $X\setminus V$ has codimension at least three in $X$ and its intersection with $E$ has codimension at least two in $E$. Both outer terms of~\eqref{proof:hodge-graded-extension} have depth at least two along $X\setminus V$, by reflexivity on $X$ and $E$, respectively. The depth lemma gives the same bound for $T^p$. Proposition~\ref{prop:simplicial-comparison} identifies $T^p|_V$ with $\Om{p}{V}{0}{0}$, and Hartogs extension gives $T^p\simeq\Om{p}{X}{0}{0}$. This proves~\eqref{proof:canonical-graded}.

\emph{Descent to the canonical modification.} On a chart admitting a $\mathfrak C$-strict sorting function, order the components of $C$ before those of $B$. Proposition~\ref{cor:ordered-stars} gives a log-simplicial modification $g\colon Z\to X$ by taking the stars in reverse order. The successive $*$-extensions compose to the direct image from $X\setminus(B+C)$ to $X\setminus B$, and the subsequent $!$-extensions compose to extension by zero to $X$. The same description on $Z$ gives $Rg_*\mathcal N_Z\simeq\mathcal N_X$ by Lemma~\ref{lem:ordered-descent}. Proper direct image of the graded de Rham complex and Lemma~\ref{lem:relative-vanishing} therefore give~\eqref{proof:canonical-graded} on this chart. These graded identifications agree on common \'etale refinements because they restrict to the ordinary de Rham identification on the common smooth open away from the boundary, and the target sheaves are torsion-free.

By Proposition~\ref{prop:canonical-model}, the preceding comparison applies on each chart of the canonical modification $f\colon Y\to X$. Taking its proper direct image and using Lemma~\ref{lem:relative-vanishing} gives
\[
 \bigl(\tDR_X(\mathcal M_X)|_{z=0}\bigr)_{-p}
 \simeq Rf_*\Om{p}{Y}{B_Y}{C_Y}[n-p]
 \simeq \Om{p}{X}{B}{C}[n-p].
\]
Here the extensions are taken on $Y$, where $Y\setminus(B_Y+C_Y)$ can strictly contain $f^{-1}(X\setminus(B+C))$. After shifting by $[-n]$, Section~\ref{sec:saito} recovers the filtered complex from these sheaves and $d_1$. On the smooth normal crossing locus, $d_1=d$ by Theorem~\ref{pre:smooth-comparison}. The targets are torsion-free, so equality extends to $X$, proving~\eqref{proof:hodge-realization}. The construction is natural under restriction and finite-group actions preserving the markings.
\end{proof}

For proper $X$, \eqref{proof:hodge-realization} and
Lemma~\ref{lem:strictness} give strictness of
$R\pi_{X,*}\tDR_{(X,!B,*C)}^\bullet$.
Equation~\eqref{pre:strictness-graded} gives degeneration for the analytic complex. The map from the algebraic to the analytic spectral sequence is an isomorphism on the $E_1$ page by proper GAGA for each coherent sheaf $\Om{p}{X}{B}{C}$. This proves Theorem~\ref{thm:degeneration}.

\subsection{Twisted Hodge realizations}
\label{proof:twisted-comparisons}

Following the cyclic covering construction of Esnault and Viehweg~\cite[3.11 and 3.16]{EV92}, we realize the twisted logarithmic complex as a character summand of the direct image of the untwisted complex on the cover.

Fix a compatible divisor $L$ and coefficients $\mathbf b,\mathbf c$ as in Definition~\ref{intro:compatible}, and write $(X,!F,*G,!\!*\mathbf{h}H)$ for the induced quadruple. Let
$\pi\colon \widehat X\to X$ and $\chi$ be the normalized cyclic cover and
character fixed in the construction of \eqref{intro:twisted-dr}.
Put $U=X\setminus(F+G+H)$, let $\pi_U\colon \pi^{-1}U\to U$ be the restricted cover, and set $\mathcal K_U=(\pi_{U,*}\C_{\pi^{-1}U})_\chi$. This constructible sheaf is a rank-one finite-monodromy local system on $U_{\reg}$, where the cover is \'etale. The cover can ramify at singular points of $U$.

Write $\widehat B=(\pi^{-1}(F+H))_{\mathrm{red}}$ and
$\widehat C=(\pi^{-1}G)_{\mathrm{red}}$. To apply Theorem~\ref{thm:hodge-realization} to this cover, work analytically on the transverse toric slice $S_\sigma$ times a simply connected polydisc in the orbit factor. A small polydisc neighborhood in a monomial embedding of $S_\sigma$ meets the torus in a space with convex logarithmic radial coordinates, which retracts onto an orbit of the maximal compact subtorus. The cover is \'etale there, since the orders of $g$ off the toric boundary are divisible by $m$. Its monodromy $\rho\colon N_\sigma\to\mu_m$ identifies each connected component with the restriction of the toric lattice cover defined by $\ker\rho$; the components are indexed by $\mu_m/\operatorname{im}\rho$. Uniqueness of normalization extends this identification across the boundary. Thus locally $\pi$ is a disjoint union of toric lattice covers, and $(\widehat X,!\widehat B,*\widehat C)$ is toroidal by the analytic criterion for Definition~\ref{intro:triple} in Section~\ref{sec:prelim}.

Put $\widetilde K_L=\tDR^\bullet_{(X,!F,*G,!\!*\mathbf{h}H)}[-n]$ and
$\widetilde K_{L,1}=\widetilde K_L|_{z=1}$, where $n=\dim X$.
For a prime divisor $P$ of $X$ and a prime divisor $\widehat P$ over it with ramification index $r$, any rational $\chi$-section $at$ satisfies $\operatorname{ord}_{\widehat P}(at)=r(\operatorname{ord}_P(a)+\operatorname{ord}_P(g)/m)$. Thus $a\mapsto at$ identifies
\[
 \cO_X\bigl(\lfloor\operatorname{div}(g)/m\rfloor\bigr)
 \simeq (\pi_*\cO_{\widehat X})_\chi.
\]
The coefficients fixed in Definition~\ref{intro:compatible} give
\[
 \lfloor\operatorname{div}(g)/m\rfloor-F=L-B,\qquad
 \{\operatorname{div}(g)/m\}=(\mathbf1-\mathbf h)H .
\]
In codimension one, the normalized cover is locally a power map in the transverse coordinate. The smooth cyclic-cover calculation and logarithmic comparison~\cite[Proposition~3.14, Lemma~4.1 and~(4.4)]{Wei23b} then give the $\chi$-lattice $\Omega_X^p(\log(B+C))(L-B)$ with connection~\eqref{intro:cyclic-connection}. The $!$-condition upstairs requires positive order over $F+H$. Along $F$, it subtracts $F$ from the integral part. Along $H$, the fractional part is nonzero, so every allowed eigensection already has positive order upstairs. For $m=1$, the identification is multiplication by $t=g$.

Finite direct image preserves the Hartogs extension property, since the inverse image of a subset of codimension at least two has codimension at least two. Hence the degree-$p-n$ term of the character summand is $\Tw{p}{X}{B}{C}{L}z^{-p}\C[z]$, and the differentials agree everywhere by torsion-freeness. Put $\mathcal M_L=(R\pi_*\mathcal M_{\widehat X})_\chi$ and $\widetilde\pi=\pi\times\id_{\mathbb A^1}$. The comparison on $\widehat X$ in Theorem~\ref{thm:hodge-realization} is deck-equivariant, so proper direct image gives
\begin{equation}\label{proof:kummer-eigencomplex}
 \tDR_X(\mathcal M_L)
 \simeq\bigl(R\widetilde\pi_*\tDR^\bullet_{(\widehat X,!\widehat B,*\widehat C)}\bigr)_\chi
 \simeq\widetilde K_L[n].
\end{equation}
Over $\pi^{-1}U$, the marked boundary is empty and the canonical modification is the identity. Thus~\eqref{proof:kummer-eigencomplex} at $z=1$ gives $\widetilde K_{L,1}|_U\simeq\mathcal K_U$ on all of $U$.

\begin{lemma}\label{proof:twisted-coordinate-extensions}
Let $U_{\mathrm{et}}\subseteq X$ be the largest open set over which $\pi$ is \'etale, and let
$j\colon U_{\mathrm{et}}\hookrightarrow X$. For $K=\widetilde K_{L,1}$ in $D^b_c(X^{\mathrm{an}},\C)$ and for $K=\mathcal M_L[-n]$ in $D^b\MHM(X)$, the adjunction maps are isomorphisms
\begin{equation}\label{proof:clean-etale}
 j_!j^*K\xrightarrow{\ \sim\ } K\xrightarrow{\ \sim\ } Rj_*j^*K.
\end{equation}
On $U_{\mathrm{et}}$, there is a Rees identification
\[
 j^*\widetilde K_L\simeq
 \tDR^\bullet_{(U_{\mathrm{et}},!F,*G)}[-n]\otimes_{\C}\mathcal L_\chi,
\]
where $\mathcal L_\chi$ is the rank-one local system obtained by restricting $(\pi_*\C_{\widehat X})_\chi$ to $U_{\mathrm{et}}$. In particular, every component of $H$ is clean for
$\widetilde K_{L,1}$, including its intersections with other strata.
\end{lemma}

\begin{proof}
On a smooth normal crossing chart, the cover is \'etale off $H$. The constant Hodge module is strictly non-characteristic along the coordinate strata. The residues $1-h_i$ are nonintegral, so~\cite[Lemma~4.2]{Wei23b} gives clean extension across $H$, with the $F$- and $G$-extensions in the remaining coordinates.

At a singular point $x\notin U_{\mathrm{et}}$, we may pass to the transverse slice, since the smooth polydisc factor preserves stalk and costalk vanishing. By~\eqref{proof:kummer-eigencomplex} at $z=1$, we take the $\chi$-summand on the local cover. The deck group permutes its components transitively with stabilizer $I=\operatorname{im}\rho$, so a $\chi$-eigenclass is determined by its restriction to one component, where it transforms by $\chi|_I$. Here $x\notin U_{\mathrm{et}}$ gives $I\ne1$, so faithfulness makes $\chi|_I$ nontrivial.

On this component, apply the slice calculation of Section~\ref{pre:twistor-dr} in the lattice $\ker\rho$ with the markings $\widehat B,\widehat C$. Let $y$ be its vertex and $V$ a monomial neighborhood as in that calculation. On $V$ and its intersections with face charts, $I$ acts by torus translations and hence trivially on the invariant forms computing hypercohomology. The $\chi|_I$-part therefore vanishes on each of these opens. The intersections of $V$ with the facet charts cover $V\setminus\{y\}$, and each finite intersection is $V$ intersected with the chart of the common face. Their \v Cech complex gives the same vanishing on $V\setminus\{y\}$. Shrinking $V$ gives vanishing of the $\chi|_I$-part of the stalk at $y$. The relative cohomology sequence of $(V,V\setminus\{y\})$ gives the corresponding costalk vanishing. Finite proper base change gives both vanishings at $x$. On $X\setminus U_{\mathrm{et}}$, the ordinary restriction has zero stalks and the extraordinary restriction has zero costalks. The latter vanishes by applying the stalk criterion to its Verdier dual. The localization triangles now give~\eqref{proof:clean-etale}. The same follows for $\mathcal M_L[-n]$, since its complex realization is $\widetilde K_{L,1}$ by~\eqref{proof:kummer-eigencomplex} and realization detects isomorphisms.

Every component of $H$ lies in $X\setminus U_{\mathrm{et}}$ because it has nontrivial generic inertia. The same stalk and costalk vanishings give clean extension across that component. On an evenly covered open subset of $U_{\mathrm{et}}$, the Rees complex upstairs is a copy of the logarithmic Rees complex for $F,G$ on each sheet. Taking the $\chi$-summand tensors the downstairs complex with $\mathcal L_\chi$, as asserted.
\end{proof}

The open set $U_{\mathrm{et}}$ in Lemma~\ref{proof:twisted-coordinate-extensions} is the Cartier locus of the integral Weil divisor $D=(L-\mathbf bB+\mathbf cC)|_{X\setminus H}$. Indeed, $U_{\mathrm{et}}\subseteq X\setminus H$. On $X\setminus H$, we have $\operatorname{div}(g)=mD$, so the valuation formula above identifies $(\pi_*\cO_{\widehat X})_\chi$ with $\cO(D)$, which is invertible wherever $\pi$ is \'etale. Conversely, if locally $D=\operatorname{div}(h)$, the normalized cover is given by $(t/h)^m=g/h^m$ with $g/h^m$ a unit, hence is \'etale. Thus $L$ and the chosen coefficients determine $U_{\mathrm{et}}$.

\begin{theorem}\label{thm:twisted-hodge}
Let $(X,!B,*C)$ be a toroidal triple of dimension $n$ with a compatible divisor $L$. Fix coefficients giving the quadruple $(X,!F,*G,!\!*\mathbf{h}H)$, and fix the cyclic cover and character above. Let $\mathcal M_L=(R\pi_*\mathcal M_{\widehat X})_\chi\in D^b\MHM(X)$. Then
\[
\tDR_X(\mathcal M_L)\simeq
\tDR_{(X,!F,*G,!\!*\mathbf{h}H)}^\bullet.
\]
Its homogeneous pieces at $z=0$ are
\[
\bigl(\tDR_X(\mathcal M_L)|_{z=0}\bigr)_{-p}
 \simeq\Tw{p}{X}{B}{C}{L}[n-p].
\]
If $X$ is proper, then $R\pi_{X,*}\tDR_{(X,!F,*G,!\!*\mathbf{h}H)}^\bullet$ is strict. If the quadruple is plenary and $F=G=0$, then $\mathcal M_L$ is a pure Hodge module of weight $n$.
\end{theorem}

\begin{proof}
Taking internal degree $-p$ at $z=0$ in the comparison~\eqref{proof:kummer-eigencomplex} gives the graded piece. For proper $X$, strictness from Lemma~\ref{lem:strictness} passes to the character summand.

In the plenary case with $F=G=0$, the open set $U=X\setminus H$ is smooth. The Cartier-locus description above gives $U_{\mathrm{et}}=U$. Write $j\colon U\hookrightarrow X$. The restriction $j^*\mathcal M_L$ is the $\chi$-summand of the finite \'etale direct image of the shifted constant variation on $\pi^{-1}U$. It is a polarizable variation of type $(0,0)$ shifted by $n$. By Lemma~\ref{proof:twisted-coordinate-extensions}, the canonical map $j_!j^*\mathcal M_L\to Rj_*j^*\mathcal M_L$ factors through $\mathcal M_L$ by two isomorphisms in $D^b\MHM(X)$. The source has perverse cohomology in degrees $\leq0$ and the target in degrees $\geq0$. Thus $\mathcal M_L$ is the intermediate extension of the variation and is pure of weight $n$ by~\cite[Theorem~3.21]{Sai90}.
\end{proof}

\subsection{Kawamata--Viehweg vanishing for toroidal triples}\label{sec:vanishing}\label{subsec:vanishing}

For any $n$-dimensional toroidal triple and any compatible $L$ with its fixed cover and character, we have
\begin{equation}\label{van:upper-perverse}
 \dR{X}{B}{C},\quad \widetilde K_{L,1}[n]\ \in\ {}^pD^{\leq0}(X).
\end{equation}
In a toric chart, take the product $S_\sigma\times P$ of Section~\ref{sec:prelim}, with $P$ a polydisc in $O(\sigma)$ and $d=\dim\sigma$. Write $\widetilde K_{\sigma,1}$ for the unshifted logarithmic complex on $S_\sigma$. Since the markings come from $S_\sigma$, the holomorphic Poincar\'e lemma in the $P$-directions gives
\[
 \bigl(\dR{X}{B}{C}[-n]\bigr)|_{S_\sigma\times P}
 \simeq \widetilde K_{\sigma,1}\boxtimes\C_P .
\]
The slice complex has terms in degrees $0,\ldots,d$. Thus the $b$th cohomology sheaf of $\dR{X}{B}{C}[-n]$ meets only orbits of codimension $d\geq b$, which is the support condition for the upper perverse bound. The twisted fiber is a summand of the finite direct image of the untwisted fiber on the Kummer cover by~\eqref{proof:kummer-eigencomplex}. Finite direct image is exact on constructible sheaves and preserves dimensions of supports, so the same bound holds for this summand.

\begin{proof}[Proof of Theorem~\ref{thm:positive}]
The stronger positive vanishing in the toric case is Theorem~\ref{thm:simplicial-bott}. For the toroidal assertions, we induct on $n$, the case $n=0$ being immediate. Choose $\ell>1$ with $\ell A$ very ample and take a general $E\in|\ell A|$. By Bertini, $E$ is transverse to the inverse images of the torus orbits in a finite collection of common \'etale charts covering $X$. At each point of $E$, its local equation can therefore serve as one of the smooth orbit coordinates of Section~\ref{sec:prelim}. The marked germ of $X$ is the product of the marked germ of $E$ with a disk. Thus $(E,!B_E,*C_E)$ is toroidal, where $B_E=(B\cap E)_{\mathrm{red}}$ and $C_E=(C\cap E)_{\mathrm{red}}$.

To restrict the Weil twist, put $D=L-\mathbf bB+\mathbf cC$. Since $D\sim_\Q\pm A$, we may write $mD=\operatorname{div}(g)$ locally for a sufficiently divisible $m>0$. By Section~\ref{proof:twisted-comparisons}, the normalized cover $t^m=g$ is pulled back from the transverse toric slice. The eigensheaf for the character $\zeta\mapsto\zeta$ in the direct image of its structure sheaf is $\cO_X(\lfloor D\rfloor)$. The divisor $L-\lfloor D\rfloor=\lceil\mathbf b\rceil B-\lfloor\mathbf c\rfloor C$ also comes from the slice. Hence $\cO_X(L)$ is locally pulled back from $\cO_E(L_E)=(\cO_X(L)\otimes\cO_E)^{\vee\vee}$ in the product with the disk. The product assertion of Proposition~\ref{lem:invariant-twisted-restriction} gives restriction~\eqref{proof:invariant-twisted-restriction} and residue~\eqref{proof:twisted-residue-coordinate} along $E$ with this coefficient. The maps glue since they agree generically on $E$ with restriction and residue and their targets are torsion-free on $E$. Moreover, $E\cap X_{\reg}$ contains every codimension-one point of $E$. Restricting the original relation there gives $L_E\sim_\Q\pm A|_E+\mathbf b B_E-\mathbf c C_E$ with the same boundary coefficients.

In the positive case, $p+q>n$ implies $H^q(E,\Tw{p}{E}{B_E}{C_E}{L_E})=0$ by induction, with zero form sheaf for $p>\dim E$. In the negative case, the product description of the twist gives $U_{L_E}=U_L\cap E$. Locally along this locus, $U_L$ is the product of $U_{L_E}$ with a disk. Since pullback by the projection shifted by $[1]$ preserves and detects perversity, the hypothesis on $\C_{U_L}[n]$ gives perversity of $\C_{U_{L_E}}[n-1]$. For $p,q\geq1$ and $p+q<n$, induction gives $H^{q-1}(E,\Tw{p-1}{E}{B_E}{C_E}{L_E})=0$, since $(p-1)+(q-1)<n-1$. For $q=0$, this cohomology group is zero, and for $p=0$, the residue term is zero. Thus restriction~\eqref{proof:invariant-twisted-restriction} reduces the positive assertion to the same vanishing with $B+E$ in place of $B$, and residue~\eqref{proof:twisted-residue-coordinate} reduces the negative assertion to the same vanishing with $C+E$ in place of $C$.

Substituting $A\sim_\Q E/\ell$ makes $L$ compatible with the enlarged boundary in each case. Fix its cyclic cover and character, and let $K=\widetilde K_{L,1}[n]$ be the complex realization of $\mathcal M_L$ in Theorem~\ref{thm:twisted-hodge}. The coefficient of $E$ in $\mathbf hH$ is $1/\ell$ or $1-1/\ell$, respectively. For $j\colon V=X\setminus E\hookrightarrow X$, Lemma~\ref{proof:twisted-coordinate-extensions} gives $j_!j^*K\simeq K\simeq Rj_*j^*K$.

In the positive case, $j^*K\in{}^pD^{\leq0}(V)$ by~\eqref{van:upper-perverse}. In the negative case, all the boundary coefficients lie in $(0,1)$, so $F=G=0$ and $H=B+C+E$. The \'etale locus of the cover is $U_{\mathrm{et}}=U_L\setminus E$ by the Cartier-locus description in Section~\ref{proof:twisted-comparisons}. On $U_{\mathrm{et}}$, the complex $K$ is $\mathcal L_\chi[n]$, which is perverse since $\C_{U_L}[n]$ is perverse and $\mathcal L_\chi$ is locally constant of rank one. For $k\colon U_{\mathrm{et}}\hookrightarrow V$, Lemma~\ref{proof:twisted-coordinate-extensions} gives $j^*K\simeq Rk_*\mathcal L_\chi[n]$. Since $Rk_*$ preserves the lower perverse bound, $j^*K\in{}^pD^{\geq0}(V)$.

We now apply the argument following~(5.4) in~\cite[proof of Theorem~1.3]{Wei23b}. By clean extension across $E$ and properness of $X$, we may compute $\mathbb H^r(X,K)$ on the affine open $V$ using ordinary cohomology in the positive case and compactly supported cohomology in the negative case. The respective perverse bounds therefore give vanishing for $r>0$ and $r<0$ by Artin vanishing~\cite[Theorem~3.59]{MS21}. By proper strictness in Theorem~\ref{thm:twisted-hodge}, the twisted form of~\eqref{pre:strictness-graded} identifies the required coherent cohomology group with $\Gr^F_{-p}\mathbb H^{p+q-n}(X,K)$ via proper GAGA. This proves both assertions.
\end{proof}

\begin{example}\label{ex:negative-cone}
The lower vanishing can fail for $B=C=0$. Let $X$ be the projective cone over $S=\mathbb P^1\times\mathbb P^2$ embedded by $\cO_S(1,1)$, and let $v$ be its vertex. Then $\dim X=4$ and
\[
 H^2\bigl(X,\Omega_X^{[1]}\otimes\cO_X(-1)\bigr)\simeq\C,
 \qquad \Omega_X^{[1]}=(\Omega_X^1)^{\vee\vee}.
\]
Indeed, the residues along the invariant prime divisors $D_\rho$ give an exact sequence
\[
 0\longrightarrow\Omega_X^{[1]}\longrightarrow\cO_X^{\oplus4}
 \longrightarrow\bigoplus_\rho\cO_{D_\rho}\longrightarrow\C_v\longrightarrow0.
\]
The kernel follows from~\eqref{pre:weight-formula}, and the residue map is surjective on the smooth locus $X\setminus\{v\}$. On the chart of the vertex cone $\sigma$, in character degree $u\in\sigma^\vee\cap M$, the map evaluates covectors on the ray generators annihilated by $u$. For $u\ne0$, these generators belong to a simplicial proper face and are linearly independent. For $u=0$, all five rays occur, corresponding to the five invariant prime divisors of $S$, and their span has dimension four. This gives the stated cokernel.

For $Y=X$ or any $D_\rho$ through $v$, the base $Y\cap S$ of the projective cone $Y$ is a hyperplane section. Both $Y$ and $Y\cap S$ are complete toric varieties, so their structure sheaves have only constant global sections and no higher cohomology~\cite[Theorem~9.2.3]{CLS11}. The hyperplane sequence gives $H^j(Y,\cO_Y(-1))=0$ for all $j$. The remaining invariant divisor is $S$, on which $\cO_S(-1,-1)$ likewise has zero cohomology. After twisting the residue sequence by $\cO_X(-1)$, its two middle terms therefore have zero cohomology. Splitting at the image of the residue map gives the asserted isomorphism.

Take $L=-A$, where $A$ is a hyperplane section of $X$. Then $U_L=X$. Since $1+2<4$, the nonvanishing above implies by Theorem~\ref{thm:positive}\textup{(\ref{intro:vanishing-negative})} that $\C_X[4]$ is not perverse.
\end{example}

\section{Boundary exchange, logarithmic comparison and perversity}\label{sec:circuits}

Put $U=X\setminus(B+C)$. Theorem~\ref{thm:hodge-realization} realizes the logarithmic complex by taking boundary extensions on the canonical modification and then their proper direct image. Here we determine when the two extension orders on $X$ agree and which orders recover the logarithmic complex. On a transverse toric slice, the charts of the maximal faces avoiding a given divisor compute localization along that divisor. The charts of all facets compute the costalk at the vertex and give the criterion for perversity.

\subsection{The exchange condition and partial sortedness}

For a finite poset $P$, the vertices of $\Delta(P)$ are the elements of $P$. Its $k$-simplices are the strict chains $\eta_0<\cdots<\eta_k$, and their faces are obtained by deleting entries. A finite simplicial complex $K$ is \emph{rationally acyclic} if $H^0(K;\Q)\simeq\Q$ and $H^i(K;\Q)=0$ for every $i>0$. By extension of scalars, the same criterion holds with $\C$ in place of $\Q$.

\begin{definition}\label{def:exchange}
For a fan triple $(\mathfrak X,!\mathfrak B,*\mathfrak C)$, call a face \emph{unmarked} if all its rays belong to the unmarked ray set $\mathfrak A$. This includes the zero face. For $\mathfrak{B}_{\sigma}\ne\varnothing$ and an unmarked face $\alpha\preceq\sigma$, put
\begin{equation*}
 \mathcal P_{\sigma,\alpha}
 =\{\eta\preceq\sigma:\alpha\preceq\eta,\ \mathfrak{C}_{\eta}=\varnothing,\ \mathfrak{B}_{\eta}\ne\varnothing\},
\end{equation*}
ordered by inclusion. The triple satisfies the \emph{exchange condition} if $\Delta(\mathcal P_{\sigma,\alpha})$ is rationally acyclic for every such pair with $\mathfrak C_\sigma\ne\varnothing$. When $\mathfrak C_\sigma=\varnothing$, the poset already has largest element $\sigma$. For a toroidal triple, impose the exchange condition in every local toric chart.
\end{definition}

For a triple, the exchange condition tests cones with both markings, while partial sortedness concerns the cones without a $\mathfrak B$-ray. To allow for compatible twists, we define partial sortedness for a fan quadruple. For a triple, we take $\mathfrak H=\varnothing$.

\begin{definition}\label{def:partial-sorted}\label{def:geometric-sorted}
A fan quadruple is \emph{partially sorted} if, for every cone $\sigma$
containing no ray in $\mathfrak{B}\cup \mathfrak{H}$, the unmarked rays $\mathfrak{A}_{\sigma}$ are
precisely the rays of a face of $\sigma$. Equivalently, such a cone
admits a rational linear function vanishing on $\mathfrak{A}_{\sigma}$ and strictly
negative on $\mathfrak{C}_{\sigma}$. The equivalence follows by taking the negative
of a supporting functional of the face. For a toroidal quadruple, we
require this condition in every local toric chart.
A triple is \emph{geometrically sorted} if it is partially sorted and satisfies the exchange condition. For a toroidal triple, we require both conditions in every local toric chart.
\end{definition}

Simplicial triples are geometrically sorted. Every subset of the rays of a simplicial cone spans a face, which gives partial sortedness. If $\mathfrak B_\sigma\ne\varnothing$, the face spanned by $\mathfrak A_\sigma\cup\mathfrak B_\sigma$ is the largest element of each $\mathcal P_{\sigma,\alpha}$, so $\Delta(\mathcal P_{\sigma,\alpha})$ is contractible.

Partial sortedness and the exchange condition are independent even when $B\cap C\ne\varnothing$. In the cone $\sigma$ of Example~\ref{ex:hexagon}, the rays of $v_1,v_3$ are marked for $B$ and those of $v_4,v_6$ for $C$. Every $\mathfrak B$-free face is proper and hence simplicial, so the triple is partially sorted. The two neighbors of $v_5$ in the hexagon are $v_4$ and $v_6$, both marked for $C$. Thus the ray $\alpha=\R_{\geq0}v_5$ is maximal among the $\mathfrak C$-free faces, and $\mathcal P_{\sigma,\alpha}$ is empty. The exchange condition fails.

Conversely, on $S_\sigma\times\mathbb A^1$, let $C$ be the pullback of the reduced divisor with rays $v_4,v_6$, and set $B=S_\sigma\times\{0\}$. The $\mathfrak B$-free cone $\sigma\times\{0\}$ now has four unmarked rays. Every proper face of this cone has at most two rays, so partial sortedness fails. Write $b$ for the ray of the second factor. On each cone $\tau\times b$ containing both markings and for every unmarked face $\alpha\times\{0\}$ of that cone, $\alpha\times b$ is the least element of the exchange poset. Hence the exchange condition holds.

\subsection{Localization of the logarithmic complex}\label{sec:comparison-localization}

Put $n=\dim X$ and $\widetilde K_1=\dR{X}{B}{C}[-n]$, so logarithmic $p$-forms lie in degree $p$. Fix reduced subdivisors $B^\sharp\subseteq B$ and $C^\flat\subseteq C$. Put $B^\flat=B-B^\sharp$ and work on $W=X\setminus B^\flat$. We ask when the localization unit
\begin{equation}\label{cohom:profile-localization}
 \widetilde K_1|_W\longrightarrow (\widetilde K_1|_W)[*C^\flat]
\end{equation}
is an isomorphism. On each toric chart, we use the corresponding ray sets $\mathfrak B^\sharp,\mathfrak B^\flat,\mathfrak C^\flat$, retaining $\mathfrak A$ for the rays unmarked in the original triple.

\begin{definition}\label{def:relative-geometric-sorted}
A fan triple is \emph{geometrically $(\mathfrak B^\sharp,\mathfrak C^\flat)$-sorted} if every cone $\sigma$ with $\mathfrak B^\flat_\sigma=\varnothing$ satisfies the applicable condition below.
\begin{enumerate}
\item If $\mathfrak B_\sigma=\varnothing$, some $\mathfrak C^\flat$-free face of $\sigma$ contains all the unmarked rays $\mathfrak A_\sigma$.
\item If $\mathfrak B_\sigma\ne\varnothing$, let $\alpha\preceq\sigma$ contain no ray of $\mathfrak B\cup\mathfrak C^\flat$. Suppose that $\alpha$ is the smallest face of $\sigma$ containing $\mathfrak A_\alpha$, the set of its unmarked rays. We include $\alpha=0$. The complex of strict chains in the poset
\[
 \{\eta:\alpha\preceq\eta\preceq\sigma,\quad
       \mathfrak C^\flat_\eta=\varnothing,\quad
       \mathfrak B^\sharp_\eta\ne\varnothing\},
\]
ordered by inclusion, is rationally acyclic.
\end{enumerate}
A toroidal triple is \emph{geometrically $(B^\sharp,C^\flat)$-sorted} if the condition holds in every local toric chart. For a quadruple induced by a compatible $L$, with fixed coefficients and Kummer cover, use $(F,G)$ in place of $(B,C)$ and take reduced subdivisors $F^\sharp\subseteq F$ and $G^\flat\subseteq G$ in place of $B^\sharp$ and $C^\flat$. Test the condition only on cones whose orbits lie in the \'etale locus $U_{\mathrm{et}}$ of the cover.
\end{definition}

A face $\alpha$ in (2) may contain rays of $\mathfrak C-\mathfrak C^\flat$. When $C^\flat=C$, (2) tests exactly the unmarked faces and the poset is $\mathcal P_{\sigma,\alpha}$. On a $\mathfrak B$-free cone, (1) then says that the unmarked rays are precisely the rays of a face. Thus $(0,C)$-sortedness is partial sortedness, and $(B^\sharp,C)$-sortedness is geometric sortedness on $X\setminus B^\flat$.

For a compatible $L$, write $\widetilde K_{L,1}$ for the unshifted twisted fiber of Section~\ref{proof:twisted-comparisons}. The \'etale locus in Definition~\ref{def:relative-geometric-sorted} is the Cartier locus of $(L-\mathbf bB+\mathbf cC)|_{X\setminus H}$, as observed after Lemma~\ref{proof:twisted-coordinate-extensions}. On a connected local lattice cover $N'\subseteq N$, the orbit of $\sigma$ lies there exactly when $N_\sigma\subseteq N'$. The $(0,G)$-test concerns orbits in $U_{\mathrm{et}}\setminus F\subseteq X\setminus(F+H)$. Their cones have neither an $F$-ray nor an $H$-ray, so partial sortedness of the quadruple gives the face required in~(1).

\begin{proposition}\label{cohom:cech}
An untwisted toroidal triple is geometrically $(B^\sharp,C^\flat)$-sorted if and only if~\eqref{cohom:profile-localization} is an isomorphism. For the quadruple induced by a compatible $L$, put $W=X\setminus(F-F^\sharp)$. The quadruple is geometrically $(F^\sharp,G^\flat)$-sorted if and only if the localization unit $\widetilde K_{L,1}|_W\to(\widetilde K_{L,1}|_W)[*G^\flat]$ is an isomorphism.
The restrictions $\widetilde K_1|_B$ and $\widetilde K_{L,1}|_F$ vanish for every triple and every induced quadruple, respectively.
\end{proposition}

To test the localization map~\eqref{cohom:profile-localization} at an orbit in $W$, we compare the hypercohomology of the logarithmic complex on $V$ and on $V\setminus C^\flat$ as $V$ shrinks to the vertex of the transverse toric slice. The charts of the maximal $\mathfrak C^\flat$-free faces cover the complement of $C^\flat$ in the slice, so we can compute this restriction map using the \v Cech complex of invariant forms on the charts and their intersections.

Fix a cone $\sigma$ of an untwisted triple and put $d=\dim\sigma$. For each $\eta\preceq\sigma$, the coefficient of the invariant $q$-forms in~\eqref{pre:weight-formula} is
\[
 E_{\sigma,\eta}^q=
 \begin{cases}
 \bigwedge^q \mathfrak{A}_{\eta}^{\perp},&\mathfrak{B}_{\eta}=\varnothing,\\
 0,&\mathfrak{B}_{\eta}\ne\varnothing.
 \end{cases}
\]
Here $\mathfrak A_\eta^\perp=\bigcap_{\rho\in\mathfrak A_\eta}v_\rho^\perp\subseteq M_\sigma$, with the empty intersection equal to $M_\sigma$.
Write $\widetilde K_{\sigma,1}$ for the unshifted logarithmic complex on $S_\sigma$. Near a point of the orbit, the marked germ is a product of $(S_\sigma,x_\sigma)$ with a polydisc, and the unshifted logarithmic complex is the external product of $\widetilde K_{\sigma,1}$ with the constant sheaf on the polydisc, as in the proof of~\eqref{van:upper-perverse}. Thus we compute on $(S_\sigma,x_\sigma)$.

Let $U_{\sigma,\eta}$ be the affine chart of $\eta$ in $S_\sigma$, and let $V$ range over the monomial neighborhoods of $x_\sigma$ from Section~\ref{pre:twistor-dr}. The slice calculation there identifies the inclusion of invariant forms as a quasi-isomorphism
\begin{equation}\label{cohom:invariant-model}
 \bigoplus_qE_{\sigma,\eta}^q[-q]
 \longrightarrow R\Gamma(U_{\sigma,\eta},\widetilde K_{\sigma,1})
\end{equation}
and likewise on $V\cap U_{\sigma,\eta}$. Taking $\eta=\sigma$ and shrinking $V$ gives $H^q((\widetilde K_{\sigma,1})_{x_\sigma})=E_{\sigma,\sigma}^q$.

Let $\tau_1,\ldots,\tau_s$ be faces of $\sigma$. Write $j$ for the inclusion of the union of their charts in $S_\sigma$ and $i$ for the inclusion of the closed complement. Put $\tau_\varnothing=\sigma$ and $\tau_I=\bigcap_{k\in I}\tau_k$ for $I\ne\varnothing$. On $V$, augment the \v Cech complex of the opens $V\cap U_{\sigma,\tau_k}$ by $R\Gamma(V,\widetilde K_{\sigma,1})$. The localization triangle for the complementary immersions $i$ and $j$ identifies the direct limit over shrinking $V$ with the stalk of $i_*i^!\widetilde K_{\sigma,1}$ at $x_\sigma$. By~\eqref{cohom:invariant-model}, it is represented by the direct sum of the rows
\begin{equation}\label{cohom:augmented}
 0\longrightarrow E_{\sigma,\sigma}^q\longrightarrow
 \bigoplus_i E_{\sigma,\tau_i}^q\longrightarrow
 \bigoplus_{i<j}E_{\sigma,\tau_i\cap\tau_j}^q\longrightarrow\cdots .
\end{equation}
The summand indexed by $I$ has row degree $a=|I|$ and total degree $a+q$, including $a=0$ for the augmentation. For $J=\{j_0<\cdots<j_r\}$, the component from $I=J\setminus\{j_k\}$ to $J$ is $(-1)^k$ times the inclusion of coefficient spaces. All other components are zero.

For localization along $C^\flat$, take $\tau_1,\ldots,\tau_s$ to be the maximal $\mathfrak C^\flat$-free faces. Their charts cover $S_\sigma\setminus C^\flat$.

\begin{lemma}\label{cohom:face-rows}
For each cone $\sigma$ tested in Definition~\ref{def:relative-geometric-sorted}, its face condition holds if and only if~\eqref{cohom:augmented} is exact for every $0\leq q\leq d$. Thus geometric $(B^\sharp,C^\flat)$-sortedness is equivalent to \emph{cohomological $(B^\sharp,C^\flat)$-sortedness}, meaning exactness of these rows on all the tested cones.
\end{lemma}

\begin{proof}
We first separate the exterior coefficients by the span of the unmarked rays. Let $V_I\subset N_\sigma\otimes_\Z\C$ be the complex span of the unmarked ray generators of $\tau_I$. For each such span $V$, let $D_V^\bullet$ have a copy of $\C$ in degree $|I|$ for every $\mathfrak B$-free $\tau_I$ with $V_I=V$, with the alternating differential between these summands. Since $I\subseteq J$ implies $V_J\subseteq V_I$, the terms of row $q$ with $\dim V_I\leq k$ form a subcomplex. Its quotient by the terms with $\dim V_I<k$ is
\[
 \bigoplus_{\dim V=k}D_V^\bullet\otimes_\C\bigwedge^q V^\perp.
\]
Thus exactness of the $D_V^\bullet$ implies exactness of every row. Conversely, assume the rows are exact and argue by increasing $k$. In row $q=d-k$, all terms with $\dim V_I>k$ vanish, and the subcomplex for $\dim V_I<k$ is exact by induction. The displayed quotient is therefore exact. Since each $\bigwedge^{d-k} V^\perp$ in it is a line, every $D_V^\bullet$ with $\dim V=k$ is exact.

For $S\subseteq\mathfrak A_\sigma$, let $C_S^\bullet$ have a copy of $\C$ in degree $|I|$ for each $\mathfrak B$-free $\tau_I$ containing $S$, with the alternating differential. A ray of $\sigma$ whose generator belongs to $V_I$ lies in $\tau_I$, since $\tau_I$ is a face. Thus the same filtration of $C_S^\bullet$ has graded summands $D_V^\bullet$ for precisely the contributing spans $V$ containing the generators of $S$. Again, exactness of these summands implies exactness of $C_S^\bullet$. Conversely, assume all $C_S^\bullet$ are exact and argue by decreasing $\dim V$. For $V=V_I$ with $\tau_I$ $\mathfrak B$-free, take $S=\mathfrak A_{\tau_I}$. It spans $V$, so the lowest filtration term of $C_S^\bullet$ is $D_V^\bullet$ and all higher quotients are exact by induction. Hence $D_V^\bullet$ is exact.

Let $T_S$ be the simplex on the maximal $\mathfrak C^\flat$-free faces containing $S$. When $\mathfrak B_\sigma=\varnothing$, $C_S^\bullet$ is its augmented cochain complex, with the augmentation in degree zero. It is exact precisely when $T_S$ has a vertex. It suffices to test $S=\mathfrak A_\sigma$, which gives condition~(1) of Definition~\ref{def:relative-geometric-sorted}.

Suppose now that $\mathfrak B_\sigma\ne\varnothing$. The augmentation term is zero. If $T_S$ has no vertex, or if the smallest face $\alpha$ containing $S$ contains a $\mathfrak B$-ray, then $C_S^\bullet=0$. Otherwise, $\alpha$ is $\mathfrak B$-free and $\mathfrak C^\flat$-free. Since $S\subseteq\mathfrak A_\alpha$, $\alpha$ is also the smallest face containing $\mathfrak A_\alpha$. Thus it occurs in~(2), and a face contains $S$ exactly when it contains $\alpha$. Let $K_S\subseteq T_S$ consist of the nonempty index sets $I$ for which $\tau_I$ contains a $\mathfrak B^\sharp$-ray. Then $C_S^\bullet$ is the relative cochain complex of $(T_S,K_S)$ with the vertices placed in degree one. Since $T_S$ is a nonempty simplex, this complex is exact precisely when $K_S$ is rationally acyclic. The largest member of each chain in the poset for $\alpha$ lies in a maximal face indexing $T_S$. The subcomplexes of chains contained in these maximal faces therefore cover the complex of chains. An intersection indexed by $I$ is empty if $I\notin K_S$ and is a cone with vertex $\tau_I$ otherwise. Thus $K_S$ and the complex of chains are either both empty or homotopy equivalent by the nerve lemma. Taking $S=\mathfrak A_\alpha$ realizes every face tested in~(2), proving the equivalence.
\end{proof}

\begin{proof}[Proof of Proposition~\ref{cohom:cech}]
At an orbit of $W=X\setminus B^\flat$, we have $\mathfrak B^\flat_\sigma=\varnothing$. The localization unit is an isomorphism on that orbit exactly when the rows~\eqref{cohom:augmented} for the maximal $\mathfrak C^\flat$-free faces of $\sigma$ are exact. Lemma~\ref{cohom:face-rows} gives the asserted face condition. Also, $E_{\sigma,\sigma}^q=0$ for every $q$ when $\sigma$ contains a $\mathfrak B$-ray, so $\widetilde K_1|_B=0$. The marked germs agree on a common \'etale chart. Every neighborhood of the closed-orbit point meets the orbit of each face, since the closure of that orbit contains the point. Thus the stalk tests include all faces of the chart cone.

For the twisted assertion, work on $W=X\setminus(F-F^\sharp)$ and write $a\colon W\cap U_{\mathrm{et}}\hookrightarrow W$. Put $P=\widetilde K_{L,1}|_{W\cap U_{\mathrm{et}}}$. By Lemma~\ref{proof:twisted-coordinate-extensions}, $\widetilde K_{L,1}|_W\simeq Ra_*P$. Open base change along $W\setminus G^\flat\hookrightarrow W$ gives $(Ra_*P)[*G^\flat]\simeq Ra_*(P[*G^\flat])$, under which the unit on $W$ is $Ra_*(P\to P[*G^\flat])$. Since $a^*Ra_*\simeq\id$, it suffices to test the unit on $W\cap U_{\mathrm{et}}$. On this open set, the twisted complex is the tensor product of the untwisted complex for $(F,G)$ with $\mathcal L_\chi$. Thus the same face test applies. On all of $U_{\mathrm{et}}$, this identification gives zero stalks on $F\cap U_{\mathrm{et}}$. Lemma~\ref{proof:twisted-coordinate-extensions} gives zero stalks outside $U_{\mathrm{et}}$, proving $\widetilde K_{L,1}|_F=0$.
\end{proof}

\begin{proof}[Proof of Theorem~\ref{thm:exchange}]
On $U$, the boundary is empty, so Danilov's Poincar\'e lemma gives $\C_U\xrightarrow{\sim}\widetilde K_1|_U$~\cite[Proposition~13.4]{Dan78}. Since $\widetilde K_1|_B=0$, the localization triangle on $X\setminus C$ gives $\widetilde K_1|_{X\setminus C}\simeq i_{B,!}\C_U$. Hence $T:=\widetilde K_1[*C]\simeq Rj_{C,*}i_{B,!}\C_U=\C_X[!B][*C]$.

By open base change, the exchange morphism is the counit $j_{B,!}j_B^*T\to T$. It is an isomorphism off $B$, and both sides vanish on $B\setminus C$. At $x\in B\cap C$, its source and the source of the localization unit $\widetilde K_1\to T$ both have zero stalk. Each map is therefore an isomorphism at $x$ exactly when $T_x=0$. For $x\in O(\sigma)$, take $B^\sharp=B$ and $C^\flat=C$. The localization triangle on the slice gives $T_x=0$ exactly when all rows~\eqref{cohom:augmented} for $\sigma$ are exact. Lemma~\ref{cohom:face-rows} and the description after Definition~\ref{def:relative-geometric-sorted} identify this exactness with rational acyclicity of $\Delta(\mathcal P_{\sigma,\alpha})$ for every unmarked $\alpha\preceq\sigma$. The orbits in $B\cap C$ correspond precisely to the cones with both markings.
\end{proof}

\begin{proof}[Proof of Theorem~\ref{thm:circuit-comparison}]
The canonical modification is an isomorphism over $U$, so the complex $\mathcal M_X\in D^b\MHM(X)$ of Theorem~\ref{thm:hodge-realization} restricts to $\C_U^H[n]$ there. Its realization has zero stalks on $B$ by Proposition~\ref{cohom:cech}. The localization triangle therefore gives $\mathcal M_X[!B]\xrightarrow{\sim}\mathcal M_X$, and on $X\setminus C$, this identifies $\mathcal M_X$ with $i_{B,!}\C_U^H[n]$. In particular, $\mathcal M_X[*C]\simeq\C_X^H[!B][*C][n]$.

Apply $[!B]$ to the localization unit along $C$. Using the preceding counit on the source and open base change on the target gives the first comparison $u$ in
\[
 \mathcal M_X\xrightarrow{\ u\ }\C_X^H[*C][!B][n]
 \xrightarrow{\ e\ }\C_X^H[!B][*C][n],
\]
where $e$ is the exchange counit. By naturality of the counit, $e\circ u$ is the localization unit along $C$. By Theorem~\ref{thm:hodge-realization}, the complex realizations of $u$ and $e\circ u$ are the two stated comparisons. This construction also gives the maps on a toroidal triple. The relative $\Proj$ in~\eqref{fan:canonical-proj} commutes with \'etale pullback, and proper and open base change identify $\mathcal M_X$ and its adjunction maps on a common \'etale chart.

Both sides of $u$ have zero stalks along $B$. On $X\setminus B$, $u$ is the localization unit along $C$. By Proposition~\ref{cohom:cech} with $(B^\sharp,C^\flat)=(0,C)$, $u$ is therefore invertible exactly when the triple is partially sorted. Applying the same proposition with $(B^\sharp,C^\flat)=(B,C)$ to the localization unit $e\circ u$ on $X$ gives the geometric sortedness criterion for the composite.
\end{proof}

For a smooth normal crossing pair, the argument of~\cite[Proposition~3.14]{Wei23b} compares the constant Hodge module's logarithmic complex with every extension order. On a toroidal triple, the components of $F$ following a given $*$-extension determine the open set where its unit must be invertible.

\begin{proposition}\label{cohom:ordered}
For the quadruple induced by $L$, fix a reduced subdivisor $G^\circ$ of $G$ and an order $\mathscr{O}$ on the components of $F+G^\circ+H$, with $!$ on $F$, $*$ on $G^\circ$, and either symbol on each component of $H$. For each irreducible component $E$ of $G^\circ$, let $F_{>E}$ be the sum of the $F$-components following $E$. Then
\begin{equation*}
 \widetilde K_{L,1}\simeq\widetilde K_{L,1}[\mathscr{O}]
\end{equation*}
if and only if the quadruple is geometrically $(F-F_{>E},E)$-sorted for every irreducible component $E$ of $G^\circ$. For an untwisted triple, take $H=0$ and replace $(\widetilde K_{L,1},F,G)$ by $(\widetilde K_1,B,C)$.
\end{proposition}

\begin{proof}
Write $\mathscr O=(\diamond_1D_1,\ldots,\diamond_sD_s)$ and $X_r=X\setminus\sum_{a>r}D_a$ for $0\le r\le s$. Assume the stated face conditions. At each inclusion $X_{r-1}\hookrightarrow X_r$, the prescribed extension of $\widetilde K_{L,1}|_{X_{r-1}}$ is $\widetilde K_{L,1}|_{X_r}$ by the counit for $!$ or the unit for $*$. Indeed, on a component of $F$, its stalks vanish, and a component of $H$ is clean by Lemma~\ref{proof:twisted-coordinate-extensions}. For $D_r=E\subset G^\circ$, Proposition~\ref{cohom:cech} gives the unit isomorphism along $E$ on $X\setminus F_{>E}$, which restricts to $X_r$. Iterating gives the required ordered extension by open base change. These maps are the complex realizations of the corresponding adjunction isomorphisms for $\mathcal M_L[-n]$.

Conversely, suppose $\widetilde K_{L,1}\simeq\widetilde K_{L,1}[\mathscr O]$. Fix a component $E$ of $G^\circ$ and put $V=U_{\mathrm{et}}\setminus F_{>E}$. This open misses $H$ and every $F$-component following $E$. On $V$, the extensions from $E$ onward therefore reduce to $*$-extensions. Open base change and composition identify their result with a derived direct image from the complement of $E$ and the later $*$-divisors in $V$. This inclusion factors through $j\colon V\setminus E\hookrightarrow V$. Thus $\widetilde K_{L,1}|_V$ lies in the image of $Rj_*$ and its unit along $E$ is invertible. By the \'etale-locus test in Proposition~\ref{cohom:cech}, this is exactly the required condition on $X\setminus F_{>E}$. In the untwisted case, take $U_{\mathrm{et}}=X$ and $H=0$.
\end{proof}

\begin{proposition}\label{prop:twisted-comparison}
Let $L$ be a compatible Weil divisor on a toroidal triple $(X,!B,*C)$ of dimension $n$. Choose coefficients giving the quadruple $(X,!F,*G,!\!*\mathbf hH)$, and fix the cyclic cover and its faithful character. Put $K=\widetilde K_{L,1}[n]=\DR^\bullet_{(X,!F,*G,!\!*\mathbf hH)}$ for the resulting twisted complex.
For either choice of $K'$ in
\[
 K[*(G+H)][!F],\qquad K[*G][!(F+H)],
\]
we have $K\simeq K'$ if and only if the induced quadruple is geometrically $(0,G)$-sorted. For either choice of $K'$ in
\[
 K[!F][*(G+H)],\qquad K[!(F+H)][*G],
\]
we have $K\simeq K'$ if and only if the quadruple is geometrically $(F,G)$-sorted. The latter condition is also equivalent to recovering $K$ in every order on the components of $F+G+H$, with $!$ on $F$, $*$ on $G$, and either symbol on each component of $H$.
The same assertions hold for $\mathcal M_L$ in $D^b\MHM(X)$.
\end{proposition}

\begin{proof}
For $D=\sum_i E_i$, successive $*$-extensions across the $E_i$ compose, by open base change, to the direct image from the complement of $D$. Thus invertibility of all the component units implies invertibility of the unit along $D$. The converse holds because the inclusion of that complement factors through the complement of each $E_i$.

Resolve each displayed block into components with the same extension symbol. For $E$ in $G$, we have $F_{>E}=F$ in the first pair of orders and $F_{>E}=0$ in the second. Proposition~\ref{cohom:ordered} therefore tests the component units on $X\setminus F$ and $X$, respectively. On each open set, these tests are equivalent to invertibility of the unit along $G$. Proposition~\ref{cohom:cech} gives the asserted conditions.

If the unit along $G$ is invertible on $X$, each component unit is invertible on every $X\setminus F_{>E}$, giving recovery in every order by Proposition~\ref{cohom:ordered}. Conversely, recovery for $K[!(F+H)][*G]$ implies invertibility of the unit along $G$ on $X$. The adjunction maps in Proposition~\ref{cohom:ordered} give the same comparisons for $\mathcal M_L$, and realization detects their invertibility.
\end{proof}

On $U=X\setminus(F+G+H)$, we have $K|_U\simeq\mathcal K_U[n]$ by~\eqref{proof:kummer-eigencomplex}. Proposition~\ref{cohom:ordered} with $G^\circ=G$ determines which orders recover $K$ from this restriction. For a geometrically sorted untwisted triple, every order recovers $\widetilde K_1$ from $\C_U$ by Proposition~\ref{prop:twisted-comparison} with $L=0$, $\mathbf b=\mathbf c=0$, and the degree-one cover.

\subsection{A criterion for perversity}\label{sec:perversity}

The charts of the facets $\tau_1,\ldots,\tau_s$ cover $S_\sigma\setminus\{x_\sigma\}$ since every proper face lies in a facet (the cover is empty for $\sigma=0$). The corresponding row~\eqref{cohom:augmented} is
\begin{equation}\label{cohom:costalk-rows}
 R_{\sigma,q}^a=\bigoplus_{|I|=a} E_{\sigma,\tau_I}^q,\qquad
 R_{\sigma,q}^0=E_{\sigma,\sigma}^q .
\end{equation}
These rows compute the costalk at $x_\sigma$ by the localization triangle for the punctured slice.

\begin{proposition}\label{cohom:perversity}
Let $(X,!B,*C)$ be a toroidal triple of dimension $n$. The complex $\widetilde K_1[n]=\dR{X}{B}{C}$ belongs to ${}^pD^{\leq0}(X)$. It is perverse if and only if, for every cone $\sigma$ in each local toric chart,
\begin{equation}\label{cohom:perverse-condition}
 H^a(R_{\sigma,q}^{\bullet})=0
 \qquad\text{for all }a,q\geq0\text{ with }a+q<\dim\sigma .
\end{equation}
For a compatible Weil divisor $L$ with its fixed Kummer cover and faithful character, $\widetilde K_{L,1}[n]$ also belongs to ${}^pD^{\leq0}(X)$. It is perverse if and only if~\eqref{cohom:perverse-condition} holds with markings $F,G$ on every cone whose orbit lies in the \'etale locus of the cover.
\end{proposition}

\begin{proof}
By~\eqref{van:upper-perverse}, only the lower perverse bound remains. Put $d=\dim\sigma$ and write $i_\sigma\colon\{x_\sigma\}\hookrightarrow S_\sigma$. The orbit has dimension $n-d$, so the cosupport condition requires the exceptional restriction of $\widetilde K_1[n]$ to this orbit to vanish in degrees $k<-(n-d)$. In the product with an orbit polydisc, this restriction is the external product of $(i_\sigma^!\widetilde K_{\sigma,1})[n]$ with the constant sheaf on the polydisc. The required unshifted slice degrees are therefore $r=k+n<d$. Applying~\eqref{cohom:augmented} to the facet cover gives
\[
 H^r(i_\sigma^!\widetilde K_{\sigma,1})\simeq
 \bigoplus_{a+q=r}H^a(R_{\sigma,q}^\bullet).
\]
This is precisely the test~\eqref{cohom:perverse-condition} on each cone.

On $U_{\mathrm{et}}$, the twisted complex is the tensor product of the untwisted complex for $(F,G)$ with $\mathcal L_\chi$ by Lemma~\ref{proof:twisted-coordinate-extensions}. Tensoring with this rank-one local system preserves costalk vanishing. The exceptional restriction of $\widetilde K_{L,1}$ to $X\setminus U_{\mathrm{et}}$ is zero by~\eqref{proof:clean-etale}, so these are all the required tests.
\end{proof}

Since $\widetilde K_1[n]$ realizes $\mathcal M_X$ by Theorem~\ref{thm:hodge-realization}, \eqref{cohom:perverse-condition} is equivalent to $\mathcal M_X\in\MHM(X)$. Theorem~\ref{thm:circuit-comparison} gives the same perversity criterion for $\C_X[*C][!B][n]$ under partial sortedness and for both orders under geometric sortedness.

The exchange condition in Theorem~\ref{thm:exchange} depends only on the marked face poset. The differentials in~\eqref{cohom:costalk-rows} also use the inclusions among the subspaces $\mathfrak A_\eta^\perp$. For $B=C=0$, Kim and Venkatesh give two cones with isomorphic face posets~\cite[Example~1.7]{KV25}. By the criterion of Musta\c t\u a and Popa~\cite[Remark~4.13]{MP22}, their computation implies that the shifted constant sheaf is perverse on exactly one of the associated toric fourfolds.

On $U_L$, the markings are empty and the unshifted logarithmic complex is $\C_{U_L}$. If $n\leq3$, the perversity hypothesis of Theorem~\ref{thm:positive}\textup{(\ref{intro:vanishing-negative})} follows from Takeuchi's toric result~\cite[Lemma~2.5]{Tak20} or directly from~\eqref{cohom:costalk-rows}. For $0<d=\dim\sigma\leq3$, the row $q=0$ is the exact augmented cochain complex of the simplex on the facets. For $q>0$, the degree-zero term vanishes since $\mathfrak A_\sigma^\perp=0$. The only remaining case with $a+q<d\leq3$ is $d=3$ and $a=q=1$. A cocycle assigns to each facet $\tau_i$ a covector in the line $\tau_i^\perp\subset M_\sigma$. The covectors agree by the cocycle condition, so they vanish since distinct facets have distinct annihilator lines. Under the strict coefficient inequalities of Theorem~\ref{thm:positive}\textup{(\ref{intro:vanishing-negative})}, Example~\ref{ex:negative-cone} thus has the smallest dimension in which negative vanishing can fail.

\end{document}